\documentclass[11pt,reqno]{amsart}
\usepackage{cmap}
\usepackage[T1]{fontenc}
\usepackage[english]{babel}
\usepackage{amsfonts,amssymb,amsthm,amsmath}
\usepackage{url}
\usepackage{enumitem}

\usepackage{secdot}

\usepackage{graphicx}
\usepackage{epstopdf}
\usepackage{a4wide}
\usepackage{xcolor}
\usepackage[colorlinks=true,linktocpage,pdfpagelabels,
bookmarksnumbered,bookmarksopen]{hyperref}
\definecolor{ForestGreen}{rgb}{0.1,0.6,0.05}
\definecolor{EgyptBlue}{rgb}{0.063,0.1,0.6}
\hypersetup{
	colorlinks=true,
	linkcolor=EgyptBlue,         
	citecolor=ForestGreen,
	urlcolor=olive
}

\usepackage[hyperpageref]{backref}

\newtheorem{theorem}{Theorem}
\newtheorem{proposition}[theorem]{Proposition}
\newtheorem{lemma}[theorem]{Lemma}

\theoremstyle{definition}

\newtheorem{remark}[theorem]{Remark}

\let\OLDthebibliography\thebibliography
\renewcommand\thebibliography[1]{
	\OLDthebibliography{#1}
	\setlength{\parskip}{1pt}
	\setlength{\itemsep}{1pt plus 0.3ex}
}

\numberwithin{equation}{section}
\numberwithin{theorem}{section}
\numberwithin{equation}{section}
\numberwithin{theorem}{section}

\newcommand{\sgn}{\operatorname{sgn}}
\newcommand{\rn}{\mathbb{R}^N}
\newcommand{\R}{\mathbb{R}}

\begin{document}

\title{Global boundedness and normalized solutions to a $p$-Laplacian equation}


\date\today
\maketitle
\centerline{\scshape Raj Narayan Dhara$^{2}$, Matteo Rizzi$^1$}


{\footnotesize
 \centerline{$^1$ Dipartimento di Matematica, Università degli Studi di Bari Aldo Moro, Bari, Italy}
\centerline{$^2$Faculty of Mathematics, Wrocław University of Science and Technology, Wroc\l{}aw, Poland}
} 
\thispagestyle{empty}


\begin{abstract}
In the paper, we prove the existence of radial solutions to
\begin{equation}\notag
-\Delta_p u+({\rm sgn}(p-s)+V(x))|u|^{p-2}u+\lambda |u|^{s-2}u=|u|^{q-2}u\qquad\text{in}\,\R^N   \\
\end{equation}
with prescribed $L^s(\R^N)$-norm, where $N\ge 3,\,p\in[2,N),\,s\in(1,p],\,q\in(p\frac{N+s}{N},\frac{Np}{N-p})$ and $V:\R^N\to\R$ is a suitable radial potential. We stress that $V$ is required to be radial but not necessarily bounded, and there are no assumptions about its sign. The case $V\equiv 0$ is also included. The proof is variational and relies on a min-max argument. A key-tool is the Pohozaev identity, which is shown to be true for any solution under quite weak assumptions about the potential $V$. This identity is proved with the aid of a new global boundedness result for subsolutions to a suitable $p$-Laplace equation.
\end{abstract}

\section{Introduction}
Let $N\ge 3$, $p\in[2,N)$, $q\in(p\frac{N+s}{N},p^*),\,p^*:=\frac{Np}{N-p}$, $s\in(1,p]$ and $\rho>0$. Under these assumptions, we want to construct radial solutions to the problem
\begin{equation}\label{main-eq}
\begin{aligned}
-\Delta_p u+{\rm sgn}(p-s)|u|^{p-2}u+V(x)|u|^{p-2}u+\lambda |u|^{s-2}u&=|u|^{q-2}u\qquad\text{in $\R^N$,}  \\
\|u\|_s=\rho ,
\end{aligned}   
\end{equation}
where $V:\R^N\to\R$ is a given radial potential and $\|u\|_s:=\left(\int_{\R^N}|u|^sdx\right)^{1/s}$.\\

Consider a generalized nonlinear Time-Dependent Schrödinger Equation (TDSE) where the kinetic energy term is governed by the $p$-Laplacian operator
$$i \frac{\partial \Psi}{\partial t} = -\Delta_p \Psi - f(|\Psi|)\Psi,$$
where $\Psi(t, x)$ is a complex-valued wave function depending on time $t$ and space $x$. In quantum mechanics and optics, we are deeply interested in solitary waves or standing waves-states whose spatial profile remains constant over time, while only their phase rotates. To find these, we make the standing wave ansatz
$$\Psi(t, x) = e^{-i\lambda t} u(x),$$
where $u(x)$ is a real-valued spatial profile and $\lambda \in \mathbb{R}$ is the frequency. This ansatz into the generalized TDSE gives us a stationary elliptic equation as in~\eqref{main-eq} for the case $s=2,\, V\equiv 0$.
While the case $s=2$ corresponds to the physical conservation of the mass of particles, the case $s=p$ corresponds to the mathematical conservation of the system's inherent nonlinear scale. 
For instance, consider  an eigenvalue equation of the form
$$-\Delta_p u = \lambda |u|^{p-2}u.$$
To find solutions, we minimize the kinetic energy $\int |\nabla u|^p \, dx$ subject to the constraint $\int |u|^p \, dx = c$. The first eigenvalue represents the generalized fundamental frequency of the system.
Exploring the space between them, that is $s \in (1, p]$, allows us to model strongly interacting, anomalous fluids while guaranteeing the system remains energetically stable and won't mathematically collapse.
In complex systems like non-Newtonian fluids flowing through porous rocks (modeled by the Porous Medium Equation (state-dependent (density)) or $p$-Laplacian (gradient-dependent (flux)), the standard laws of thermodynamics are modified. In these systems, standard Boltzmann-Gibbs statistical mechanics breaks down. Instead, the system is governed by ``non-extensive statistics'' (Tsallis statistics~\cite{Tsallis1988, tsallis2009}), where the probability density of finding a particle is not simply $|u|^2$, but is related to $u$ raised to a power $s$ to account for long-range interactions and correlations between particles. Conserving $\int |u|^s \, dx = \rho$ physically corresponds to conserving a weighted measure of the fluid's volume that accounts for its tendency to clump or spread in an anomalous way. 
In particular, when $1<s<2$, physically, this models a ``droplet'' or a localized blob of fluid with a hard edge, rather than a wave that slowly fades to zero at infinity, whereas for $s>2$ models systems that resist high densities, forcing the wave to spread out more smoothly.
In applications like nonlinear optics and Bose-Einstein condensates, the $L^s$-norm represents conserved physical quantities, such as the total number of atoms or the mass of the wave. The $L^p$-supercritical ($q>p\frac{N+s}{N}$) regime is often associated with more complex and potentially singular physical phenomena, such as self-focusing, solitons (solitary waves) and nonlinear scattering. It can also describe optical collapse, where the wave shape becomes highly concentrated. Now we shall give the mathematical set up and links to the recent mathematical studies considering the TDSE involving $p$-Laplacian operator.

In \cite{WLZL}, the authors consider problem \eqref{main-eq} with $V=0,\, s=2$ and find at least a solution $u$ in the space $X:= W^{1,p}(\R^N)\cap L^2(\R^N)$. However, we do not know whether their solution is radial or not. Due to the fact that the non-linearity $f(u):=|u|^{q-2}u-|u|^{p-2}u-\lambda |u|^{s-2}u$ depends on $u$ only and not on $x$, we expect to be able to construct radial solutions to \eqref{main-eq} at least when $V=0$, or more generally for some suitable radial potentials $V$.\\

In this paper, we will prove that, in some cases, this is true. More precisely, we assume that $V\in L^\alpha(\R^N)$ with $\alpha\in[\frac{N}{p},\infty]$. Using that $\frac{p\alpha}{\alpha-1}\in(p,p^*)$ for $\alpha\in(\frac{N}{p},\infty)$, $\frac{p\alpha}{\alpha-1}=p^*$ for $\alpha=\frac{N}{p}$ and introducing the convention $\frac{p\alpha}{\alpha-1}=p$ for $\alpha=\infty$, the H\"{o}lder inequality and the Sobolev embedding $W^{1,p}(\R^N)\subset L^{\frac{\alpha p}{\alpha-1}}(\R^N)$ give
\begin{equation}\notag
\left|\int_{\R^N} V(x)|u|^{p-2} uv dx\right|\le \|V\|_\alpha\|u\|_{\frac{p\alpha}{\alpha-1}}^{p-1}\|v\|_{\frac{p\alpha}{\alpha-1}}<\infty\qquad\forall\,u,\,v\in W^{1,p}(\R^N), 
\end{equation}
hence problem \eqref{main-eq} is variational. In other words, setting $X:=W^{1,p}(\R^N)\cap L^s(\R^N)$, the solutions to \eqref{main-eq} are the critical points of the functional
\begin{equation}
J_V(u):=\frac{1}{p}\|\nabla u\|_p^p+\frac{{\rm sgn}(p-s)}{p}\|u\|^p_p+\frac{1}{p}\int_{\R^N} V(x)|u|^p dx-\frac{1}{q}\|u\|_q^q,\qquad\forall\, u\in X,
\end{equation}
constrained to the sphere $\mathcal{S}_{r,\rho}:=\mathcal{S}_\rho\cap X_r$, where $$X_r:=W^{1,p}_r(\R^N)\cap L^s(\R^N),\qquad\mathcal{S}_{\rho}:=\{u\in X:\,\|u\|_s=\rho\}$$
and $W^{1,p}_r(\R^N)$ is the space of radial functions in $W^{1,p}(\R^N)$. Setting 
$$\|u\|:=(\|\nabla u\|_p^p+{\rm sgn}(p-s)\|u\|_p^p)^{1/p}=\begin{cases}
    \|u\|_{W^{1,p}(\rn)}\qquad\text{if}\,s\in(1,p)\\
    \|\nabla u\|_p\qquad\text{if}\,p=s,
\end{cases}
$$
where $\|u\|_{W^{1,p}(\rn)}^{p}:=\|\nabla u\|_p^p+\|u\|_p^p$.
Then our functional can be rewritten as
$$J_V(u)=\frac{1}{p}\|u\|^p+\frac{1}{p}\int_{\R^N} V(x)|u|^p dx-\frac{1}{q}\|u\|_q^q\qquad\forall\, u\in X.$$

For $u\in X$ and $t>0$, we define the scaling $u^t(x):=t^{N/s}u(tx)$. Note that this scaling preserves the $L^s(\R^N)$-norm, namely $\|u^t\|_s=\|u\|_s$. Using such a scaling, we introduce the Pohozaev functional
\begin{equation}\label{pohozaev-functional}
\begin{aligned}
P_V(u)&:=\frac{d}{dt}\bigg|_{t=1}J_V(u_t)=\frac{p(N+s)-sN}{sp}\|\nabla u\|^p_p+\frac{N(p-s)}{sp}\|u\|^p_p\\
&-\frac{N(q-s)}{sq}\|u\|^q_q+\int_{\R^N} V(x)\left(\frac{N}{s}|u|^p+|u|^{p-2}u \nabla u\cdotp x\right)dx,\qquad\forall\,u\in X.
\end{aligned}    
\end{equation}
Note that, if we assume that $\tilde{W}:=V(\cdotp)|\cdotp|\in L^{\frac{\alpha p}{p-1}}(\R^N)$, with the convention that $\frac{\alpha p}{p-1}=\infty$ if $\alpha=\infty$, the functional $P_V$ is well defined on $X$. In fact, for such a potential $V$, the H\"{o}lder inequality and the Sobolev embedding $W^{1,p}(\R^N)\subset L^{\frac{\alpha p}{\alpha-1}}(\R^N)$ yield that
\begin{equation}\notag
\left|\int_{\R^N} V(x)|u|^{p-2} u(\nabla u\cdotp x)dx\right|\le \|\tilde{W}\|_{\frac{\alpha p}{p-1}}\|u\|_{\frac{p\alpha}{\alpha-1}}^{p-1}\|\nabla u\|_p<\infty\qquad\forall\,u\in W^{1,p}(\R^N)
\end{equation}
It is known that the solutions constructed in Theorem $1.1$ of \cite{WLZL} in the case $s=2,\,V\equiv 0$ satisfy the Pohozaev identity
\begin{equation}\label{pohozaev-identity}
P_V(u)=0.  
\end{equation}
In that case, the Pohozaev identity is a consequence of Theorem $1.1$ of \cite{guedda1989quasilinear}, which relies on some regularity properties of the solution $u$, such as boundedness. In our result the Pohozaev identity directly follows by construction, without using further regularity properties of the weak solution $u\in X$. In particular, it holds true for solutions which are not necessarily bounded.\\

In our forthcoming result, we will need to impose some bounds about the norms of $V$ and $\tilde{W}$. For this purpose, for $p\in[2,N)$, we set $$S_{p,\alpha}:=\begin{cases}\inf_{u\in W^{1,p}(\R^N)\setminus\{0\}}\frac{\|u\|^p}{\|u\|_\frac{p\alpha}{\alpha-1}^p}\qquad\text{if}\,\alpha\in[\frac{N}{p},\infty),\,s\in(1,p)\,\text{or}\,\alpha=N/p,\,s=p\\
\inf_{u\in W^{1,p}(\R^N)\setminus\{0\}}\frac{\|u\|_{W^{1,p}(\rn)}^p}{\|u\|_\frac{p\alpha}{\alpha-1}^p}\qquad\text{if}\,\alpha\in(N/p,\infty),\,s=p\\
1\qquad\text{if}\,\alpha=\infty,
\end{cases}$$ 
Note that, for $\alpha\in[\frac{N}{p},\infty)$, $S_{p,\alpha}^{-1/p}$ is the best constant in the Sobolev embedding $W^{1,p}(\R^N)\subset L^\frac{p\alpha}{\alpha-1}(\R^N)$.\\

In our main result we assume that \\

$(V_1)$ $\alpha\in [\frac{N}{p},\infty],\,V\in L^\alpha(\R^N)$ is a radial function such that $\tilde{W}\in L^{\frac{\alpha p}{p-1}}(\R^N),\,\|V\|_\alpha<S_{p,\alpha}$ and
\begin{equation}
\label{cond-norm-V}
sp S_{p,\alpha}^{-\frac{\alpha p}{p-1}}\|\tilde{W}\|_{\frac{\alpha p}{p-1}}+N\max\{|q-p-s|,1\}S_{p,\alpha}^{-1}\|V\|_\alpha <\min\left\{s\left(\frac{N}{q}-\frac{N-p}{p}\right),Nq-p(N+s)\right\},
\end{equation}
holds.\\

Now we are ready to state our main Theorem.
\begin{theorem}\label{main-th-hom}
Assume that $(V_1)$ holds and $\lim_{|x|\to\infty}V(x)=0$ if $\alpha\in\{N/p,\infty\}$. Then there exists $\rho_0>0$ such that, for any $\rho\in(0,\rho_0)$, there exists a constant $K_\rho>0$ and a solution $(\lambda_\rho,u_\rho)\in (0,\infty)\times X_r$ to problem \eqref{main-eq} such that $J_V(u_\rho)=c_{V,\rho}^r>0$, where
\begin{equation}
\begin{aligned}
c_{V,\rho}^r&:=\inf_{\gamma\in\Gamma_{\rho,r}}\max_{t\in[0,1]}J_V(\gamma(t))\\
\Gamma_{\rho,r}&:=\{\gamma\in C([0,1],\mathcal{S}_{\rho,r}):\,\|\gamma(0)\|\le K_\rho,\,J_V(\gamma(1))<0\}.
\end{aligned}     
\end{equation}
Moreover, $u_\rho$ satisfies the Pohozaev identity $P_V(u)=0$ and $\rho^s\lambda_\rho\to\infty$ as $\rho\to 0^+$.
\end{theorem}
Our result works for $p=2$ too. We note that the trivial potential $V\equiv 0$ satisfies $(V_1)$, therefore Theorem \ref{main-th-hom} also applies to the homogeneous case $V\equiv 0$.\\

It is interesting to compare the solutions constructed in Theorem \ref{main-th-hom} with those constructed in Theorem $1.1$ of \cite{WLZL} for $s=2,\, V\equiv 0$. In order to do so, we set, for any $\rho\in(0,\rho_0)$,
\begin{equation}
\begin{aligned}
c_{V,\rho}&:=\inf_{\gamma\in\Gamma_{\rho}}\max_{t\in[0,1]}J_V(\gamma(t))\\
\Gamma_{\rho}&:=\{\gamma\in C([0,1],\mathcal{S}_{\rho}):\,\|\gamma(0)\|\le K_\rho,\,J_V(\gamma(1))<0\}.
\end{aligned}     
\end{equation}
Moreover, let $c_\rho^r:=c^{r}_{0,\rho}$ and $c_\rho:=c_{0,\rho}$.
\begin{remark}\label{rem-rearrangement}
\begin{itemize}
\item Note that, due to the Definition of $\inf$, we have $c^r_{V,\rho}\ge c_{V,\rho}$.
\item Due to the properties of the radial (Schwarz) rearrangement, we have $c_{V,\rho}^r=c_{V,\rho}$ if $V\le 0$, for any $\rho>0$. This is true since, if $u^*$ is the radial rearrangement of a function $u\in X$, we have
$$\|\nabla u^*\|_p\le \|\nabla u\|_p,\qquad\|u^*\|_t=\|u\|_t\qquad\forall\,t\in\{s\}\cup[p,p^*]$$
and the Hardy-Littlewood inequality yields that
$$\int_{\R^N}V|u^*|^p dx\le \int_{\R^N}V|u|^p dx\qquad\forall\,u\in X$$
if $V\le 0$. %
\end{itemize}
\end{remark}
Remark \ref{rem-rearrangement} yields that, for $V\equiv 0$ and $s=2$, our solutions have the same energy as the ones constructed in \cite{WLZL}. Therefore, it would be interesting to understand if the solutions constructed in \cite{WLZL} coincide with our solutions, or if such a solution is unique. If such a uniqueness result holds true, then the solution constructed in Theorem \ref{main-th-hom} in the case $V\equiv 0,\,s=2$ must coincide with the one of \cite{WLZL}.\\

The case $s=p,\, V\equiv0$ is treated in Theorem $1.4$ of \cite{zhang2022normalized}. However, our result is new, even in the case $V\equiv 0$, since it holds for any $s\in(1,p]$. This is possible since our proof relies on a new abstract deformation Theorem (see Theorem \ref{lemma-pseudo-grad}), in which the Banach structure of the ambient space is sufficient to construct a bounded Palais-Smale sequence of $J_V|_{\mathcal{S}_{\rho,r}}$ at level $c^r_{V,\rho}$. This is another relevant difference with \cite{WLZL}. We stress that Theorem \ref{th-bd-PS} provides an important generalization of Theorem $4.5$ of \cite{ghoussoub1993duality} to a possibly non-Hilbert setting.\\

Similar results were proved in Theorem $2$ of \cite{dhara2025normalised} and Theorem $1.4$ of \cite{PR} for a non-positive potential $V$ which is not necessarily radial. However, here we treat the more general case of a possibly sign-changing potential. Namely, we do not need the assumption $V\le 0$.\\

The idea of the proof is the following. First we construct a bounded Palais-Smale sequence $\{u_n\}\in \mathcal{S}_{\rho,r}$ of $J_V|_{\mathcal{S}_{\rho,r}}$ at level $c^r_{V,\rho}$ which almost satisfies the Pohozaev identity, that is $P_V(u_n)=o_n(1)$, then we extract a subsequence which is converging to a solution $u\in \mathcal{S}_{\rho,r}$ strongly in $X$. Due to the strong convergence in $X$, the Pohozaev identity $P_V(u)=0$ is also satisfied.\\

Finally, if $(V_1)$ is satisfied with $\alpha\in(\frac{N}{p},\infty)$, we use the compactness of the embedding $W^{1,p}_r(\R^N)\subset L^\frac{p\alpha}{\alpha-1}(\R^N)$ to prove the existence of a strongly convergent subsequence. If $\alpha=\frac{N}{p}$ or $\alpha=\infty$, compactness is recovered through a splitting Lemma which is a generalization of the one that we proved in \cite{dhara2025normalised} (see Lemma \ref{splitting-lemma}).\\

In particular, in the proof of the splitting Lemma, we need to use the fact that any weak solution $u\in X$ to 
\begin{equation}\label{lim-eq}
-\Delta_p u+{\rm sgn}(p-s)|u|^{p-2}u+\lambda |u|^{s-2}u=|u|^{q-2}u\qquad\text{in $\R^N$,}   
\end{equation}
with $\lambda>0$, satisfies the Pohozaev identity $P_0(u)=0$. This fact is known for $s=2$ (see \cite{wang2021normalized,guedda1989quasilinear}), but is new for $s\in(1,p]\setminus\{2\}$. In our paper, this fact follows as a corollary of a more general result.\\ 

More precisely, given an odd continuous function $f:\R\to\R$ satisfying the growth condition 
\begin{equation}\label{growth-f-qs}
-ct^{s-1}\le f(t)\le ct^{q-1}\qquad\forall\, t>0, 
\end{equation}
for some constant $c>0$ and a potential $V\in L^{N/p}_{loc}(\R^N)$, we say that $u\in W^{1,p}_{loc}(\R^N)$ 
is a local weak solution to 
\begin{equation}\label{eq-pohozaev-intr}
-\Delta_p u+V(x)|u|^{p-2}u=f(u)\qquad\text{in}\,\R^N,
\end{equation}
if
\begin{equation}
\int_{\R^N}|\nabla u|^{p-2}\nabla u\cdotp\nabla v\,dx+\int_{\R^N} V(x)|u|^{p-2}uv\, dx=\int_{\R^N}f(u)v \, dx\qquad\forall\, v\in C^\infty_c(\R^N). 
\end{equation}
Note that, if $V\in L^\alpha(\R^N)$ for some $\alpha\in[N/p,\infty]$, $u\in X$ is a local weak solution if and only if it is a weak solution, that is a critical point of $J_{V}$ constrained to $\mathcal{S}_\rho$. However, the notion of local weak solution is more flexible, since it requires only local integrability conditions about $V$ and $u$. In particular, it allows us to treat the case of a trapping potential.\\

In this setting, we prove the following result, which establishes that the Pohozaev identity holds for a quite general non-negative potential $V$.
\begin{theorem}\label{cor-V>0}
Assume that $2\le p<N$, $1<s\le p$, $q\in(p,p^*)$, $\beta:=\max\{2,N/p\},\,V\in L^\beta_{loc}(\R^N),\, V\ge 0$, 
and $f:\R\to\R$ is continuous, odd and fulfills \eqref{growth-f-qs}. Let $u\in X$ be a local weak solution to \eqref{eq-pohozaev-intr} such that 
\begin{equation}\label{cond-int-pot}
V(\cdotp)|\cdotp||u|^{p-1}|\nabla u|\in L^1(\R^N).    
\end{equation}
Then $u$ satisfies the Pohozaev identity 
\begin{equation}\label{pohozaev}
\frac{N-p}{p}\int_{\R^N}|\nabla u|^pdx-N\int_{\R^N}F(u) dx-\int_{\R^N}V(x)|u|^{p-2}u(x\cdotp\nabla u)dx=0, 
\end{equation}    
where $F(u):=\int_0^u f(t)dt$.
\end{theorem}
\begin{remark}
If $\lambda\in\R$, $f(u)=|u|^{q-2}u-\lambda|u|^{s-2}u-{\rm sgn}(p-s)|u|^{p-2}u$ and $u\in X$ is a weak solution to \eqref{eq-pohozaev-intr} such that $V(\cdotp)|u|^p\in L^1(\R^N)$, testing equation \eqref{eq-pohozaev-intr} with $u$ we can see that
\begin{equation}\label{test-u}
\int_{\R^N}|\nabla u|^pdx+{\sgn(p-s)}\int_{\R^N}|u|^pdx+\int_{\R^N}V(x)|u|^p dx+\lambda\int_{\R^N}|u|^s dx=\int_{\R^N}|u|^q dx.    
\end{equation}
Multiplying \eqref{test-u} by $N/s$ and subtracting to \eqref{pohozaev}, we can see that the Pohozaev identity is equivalent to $P_V(u)=0$ for weak solutions $u\in X$ to \eqref{eq-pohozaev-intr} with this particular nonlinearity.
\end{remark}
The reason of the integrability assumption about $V$ is the following. $V\in L^2_{loc}(\R^N)$ is necessary to apply Theorem $1.1$ of \cite{antonini2023interior}, while $V\in L^{N/p}_{loc}(\R^N)$ is required to give sense to treat local weak solutions.\\

The Pohozaev identity has been widely studied in the literature. See, for example, Lemma $2.4$ of \cite{zhang2022normalized} and Theorem $1$ of \cite{degiovanni2003regularity}. However, both in those results and in Theorem \ref{cor-loc-bd}, the local boundedness of the solutions is assumed to prove the Pohozaev identity. However, in Theorem \ref{cor-V>0} we do not need to explicitly assume local boundedness, since it is automatically satisfied. More precisely, we already know that, under the assumptions of Theorem \ref{cor-V>0}, any weak solution $u\in X$ to \eqref{eq-pohozaev-intr} is bounded. This follows from the following global boundedness result.
\begin{theorem}\label{th-boundedness}
Let $p\in(1,N),\,q\in(p,p^*)$ and let $u\in W^{1,p}(\R^N)$ be a weak subsolution to
$$-\Delta_p u\le u^{q-1}\qquad\text{in}\,\R^N$$
such that $u\ge 0$ in $\R^N$. Then $u\in L^\infty(\R^N)$.
\end{theorem}
Theorem \ref{th-boundedness} is a particular case of Proposition \ref{prop-boundedness}, which is a generalization of Proposition $2.4$ of \cite{PR}, which treats the case $p=2,\,N=3$ in a bounded domain. The technique used in the proof is similar to the one used in the proof of Theorem $1.3$ of \cite{pucci2024bifurcation}.\\

Note that Theorem \ref{cor-V>0} yields that the Pohozaev identity is true for weak solutions $u\in X$ to \eqref{eq-Pohozaev} in case $f(t)=|t|^{q-2}t-\lambda |t|^{s-2}t-|u|^{p-2}u,\, V=0$, with $2<p<q<p^*,\,1<s\le p$. The case $s=2$ is treated in \cite{wang2021normalized,guedda1989quasilinear}. This result is new for $s\ne 2$. Moreover, it holds for weak solutions $$u\in\mathcal{H}_k:=\left\{u\in W^{1,p}(\R^N)\cap L^2(\R^N):\int_{\R^N}|x|^k|u|^pdx<\infty\right\}$$
to \eqref{eq-Pohozaev} in case $f(t)=|t|^{q-2}t-\lambda t$, with $2<p<q<p^*$, and $V(x)=|x|^k$, for any $k>0$. This case is treated in \cite{wang2023normalized}.\\

Finally, in the case $p=2$, we will see that the Pohozaev identity holds for possibly unbounded weak solutions to \eqref{eq-pohozaev-intr} with possibly unbounded and sign-changing potential $V$. This will be done in Proposition \ref{prop-p=2}. In fact, for $p=2$, the Calder\'{o}n-Zygmund estimates a bootstrap argument (see Theorem $9.11$ of \cite{gilbarg1977elliptic}) enable us to prove that, if $V$ is locally bounded, any weak solution is locally bounded too, irrespectively of the sign of $V$. However, such regularity estimates are not available for $p\ne 2$. It is actually an interesting open problem to see if the Pohozaev identity holds for $p\ne 2$ even in the case of possibly unbounded weak solutions.\\

The plan of the paper is the following: in Section \ref{sec-global-bd} we prove the global boundedness result and the Pohozaev identity, that is Theorems \ref{th-boundedness} and \ref{cor-V>0}. In Section \ref{sec-defo} we prove our deformation Theorem (see Theorem \ref{th-bd-PS}) and apply it to construct a bounde Palias-Smale sequence. In Section \ref{sec-right-inv} we construct the right inverse of a useful possibly nonlinear operator and in Section \ref{sec-proof} we finally conclude the proof of Theorem \ref{main-th-hom}.

\section{Global boundedness and the Pohozaev identity}\label{sec-global-bd}
In this section, we will assume that $f:\R\to\R$ is a continuous function fulfilling 
\begin{equation}\label{growth-f}
|f(t)|\le c(|t|^{q-1}+|t|^{s-1})    \qquad\forall\, t\in\R,
\end{equation}
and set $F(t):=\int_0^t f(s)ds$.
\begin{theorem}\label{th-Pohozaev}
Assume that $2\le p<N$, $f:\R\to\R$ is continuous and satisfies \eqref{growth-f} and $V\in L^{N/p}_{loc}(\R^N)$. 
Assume furthermore that $u\in W^{1,p}(\R^N)\cap L^s(\R^N)$ is a local weak solution to
\begin{equation}
\label{eq-Pohozaev}
-\Delta_p u+V(x)|u|^{p-2}u=f(u)\qquad\text{in}\,\R^N
\end{equation}
satisfying \eqref{cond-int-pot} and $f(u)-V|u|^{p-2}u\in L^2_{loc}(\R^N)$. Then $u$ satisfies the Pohozaev identity
\begin{equation}\label{pohozaev}
\frac{N-p}{p}\int_{\R^N}|\nabla u|^pdx-N\int_{\R^N}F(u) dx-\int_{\R^N}V(x)|u|^{p-2}u(x\cdotp\nabla u)dx=0.  
\end{equation}    
\end{theorem}
\begin{proof}
Thanks to Theorem $1.1$ of \cite{antonini2023interior}, $|\nabla u|^{p-2}\nabla u\in H^1_{loc}(\R^N)$ and $u$ is a strong solution to \eqref{eq-Pohozaev}. Then
\begin{equation}\notag
\begin{aligned}
{\rm div}\left((x\cdotp \nabla u)|\nabla u|^{p-2}\nabla u-x\frac{|\nabla u|^p}{p}\right)&= (x\cdotp \nabla u) \Delta_p u+|\nabla u|^{p-2}\nabla u\cdotp \nabla(x\cdotp \nabla u)\\
&-\frac{N}{p}|\nabla u|^p-\frac{x}{p}\cdotp\nabla(|\nabla u|^{2\frac{p}{2}})=\\
&=(x\cdotp \nabla u) \Delta_p u+|\nabla u|^{p-2}(|\nabla u|^2+\frac{x}{2}\cdotp\nabla|\nabla u|^2)\\
&-\frac{N}{p}|\nabla u|^p-|\nabla u|^{p-2}\frac{x}{2}\cdotp\nabla |\nabla u|^2\\
&=(x\cdotp \nabla u) \Delta_p u+\frac{p-N}{p}|\nabla u|^p\in L^1_{loc}(\R^N)
\end{aligned}
\end{equation}
Let $T:\R^N\to[0,1]$ be a smooth radial function such that $T=1$ in $B_1(0)$ and $T=0$ in $\R^N\setminus\overline{B}_2(0)$. Note that 
\begin{equation}\label{est-nablaT}
|\nabla T(x)||x|\le c\qquad\forall\,x\in\R^N,
\end{equation}
for some constant $c\in (0,\infty)$. Setting, for any $R>1$, $T_R(x):=T(\frac{x}{R})$, we have
\begin{equation}\notag
\begin{aligned}
{\rm div}\left(T_R(x)(x\cdotp \nabla u)|\nabla u|^{p-2}\nabla u-T_R(x)x\frac{|\nabla u|^p}{p}\right)&= T_R(x){\rm div}\left((x\cdotp \nabla u)|\nabla u|^{p-2}\nabla u-x\frac{|\nabla u|^p}{p}\right)\\
&+\nabla T_R(x)\cdotp \left((x\cdotp \nabla u)|\nabla u|^{p-2}\nabla u-x\frac{|\nabla u|^p}{p}\right)
\end{aligned}    
\end{equation}
for a. e. $x\in\R^N$. Moreover, by the divergence Theorem
\begin{equation}\notag
\begin{aligned}
\int_{B_{2R}(0)}T_R(x)(x\cdotp\nabla u)f(u) dx&=\int_{B_{2R}(0)}T_R(x)(x\cdotp\nabla F(u))dx\\
&=\int_{B_{2R}(0)}{\rm div}(T_R(x)F(u)x)dx-\int_{B_{2R}(0)}{\rm div}(T_R(x)x)F(u)dx\\
&=-N\int_{B_{2R}(0)}T_R(x)F(u)dx+o_R(1)=-N\int_{\R^N}F(u)dx+o_R(1)
\end{aligned}
\end{equation}
as $R\to\infty$, since $F(u)\in L^1(\R^N)$.\\

As a consequence, integrating over the ball $B_{2R}$ and using \eqref{est-nablaT}, we can see that
\begin{equation}\label{int-R}
\begin{aligned}
0&=\int_{\partial B_{2R}(0)}\left(T_R(x)(x\cdotp \nabla u)|\nabla u|^{p-2}(\nabla u\cdotp \nu)-T_R(x)\frac{|\nabla u|^p}{p}(x\cdotp\nu)\right)d\sigma\\
&=\int_{B_{2R}(0)} T_R(x){\rm div}\left((x\cdotp \nabla u)|\nabla u|^{p-2}\nabla u-x\frac{|\nabla u|^p}{p}\right)dx\\
&+\int_{B_{2R}(0)} \nabla T_R(x)\cdotp \left((x\cdotp \nabla u)|\nabla u|^{p-2}\nabla u-x\frac{|\nabla u|^p}{p}\right)dx\\
&=\int_{B_{2R}(0)}T_R(x)\left((x\cdotp \nabla u) \Delta_p u+\frac{p-N}{p}|\nabla u|^p\right)dx+o_R(1)\\
&=\int_{B_{2R}(0)}T_R(x)\left((x\cdotp\nabla u)(V(x)|u|^{p-2}u-f(u))+\frac{p-N}{p}|\nabla u|^p\right)dx+o_R(1)\\
&=\int_{\R^N}V(x)|u|^{p-2}u(x\cdotp\nabla u)dx+N\int_{\R^N}F(u)dx+\frac{p-N}{p}\int_{\R^N}|\nabla u|^pdx+o_R(1)
\end{aligned}
\end{equation}
as $R\to\infty$, since $u\in W^{1,p}(\R^N)$ and $V(\cdotp)|\cdotp||u|^{p-1}|\nabla u|\in L^{\frac{N}{p-1}}(\R^N)$. Hence the result is proved.
\end{proof}
Applying Theorem \ref{th-Pohozaev}, we can prove the following result.
\begin{theorem}\label{cor-loc-bd}
Assume that $2\le p<N$, $f$ is continuous and satisfies \eqref{growth-f} and $V\in L^{\beta}_{loc}(\R^N) $ with $\beta:=\max\{N/p,2\}$. 
Assume that $u\in W^{1,p}(\R^N)\cap L^\infty_{loc}(\R^N)$ is a local weak solution to \eqref{eq-Pohozaev} satisfying \eqref{cond-int-pot}. Then $u$ satisfies the Pohozaev identity
\begin{equation}\label{pohozaev}
\frac{N-p}{p}\int_{\R^N}|\nabla u|^pdx-N\int_{\R^N}F(u) dx-\int_{\R^N}V(x)|u|^{p-2}u(x\cdotp\nabla u)dx=0, 
\end{equation}    
where $F(u):=\int_0^u f(t)dt$.
\end{theorem}
\begin{proof}
Since $u\in L^\infty_{loc}(\R^N)$ and $f$ is continuous, $f(u)\in L^\infty_{loc}(\R^N)$. Moreover, since $V\in L^2_{loc}(\R^N)$, we have $V|u|^{p-2}u\in L^2_{loc}(\R^N)$. Hence the conclusion follows directly from Theorem \ref{th-Pohozaev}.  
\end{proof}
In the case $p=2$ and $f$ satisfying \ref{growth-f}, we have the following result.
\begin{proposition}\label{prop-p=2}
Assume that $p=2$, $f:\R\to\R$ is continuous, and fulfills \eqref{growth-f} and $V\in L^{\infty}_{loc}(\R^N)$. 
Let $u\in H^1(\R^N)\cap L^s(\R^N)$ be a weak solution to \eqref{eq-Pohozaev} fulfilling \eqref{cond-int-pot}. Then $u$ satisfies the Pohozaev identity \eqref{pohozaev}.
\end{proposition}
\begin{proof}
By a bootstrap argument based on the Calder\'{o}n-Zygmund estimates, we can see that $u\in L^\infty_{loc}(\R^N)$, hence the Pohozaev identity holds thanks to Theorem \ref{cor-loc-bd}.
\end{proof}
\begin{remark}
Note that, in the situation of Proposition \ref{prop-p=2}, we do not know whether the weak solutions are globally bounded or not. However, the Pohozaev identity holds. This shows that, in some case, the local boundedness of the solution is enough for the Pohozaev identity to hold.
\end{remark}
The aim of the remaining part of the section is to prove that, if $f$ satisfies a slightly more restrictive growth condition then \eqref{growth-f}, namely there exists $c>0$ such that
\begin{equation}\label{growth-f-q}
f(t)\le ct^{q-1}\qquad\forall\, t>0,    
\end{equation}
and the potential satisfies $V\ge 0$ and certain integrability conditions, any weak solution $u\in W^{1,p}(\R^N)\cap L^s(\R^N)$ satisfies the Pohozaev identity \eqref{pohozaev}. To do so, it is enough to prove that such weak solutions are bounded. For this purpose, for $p\in(1,N)$ we set
$$S_p(\Omega):=\inf_{u\in D^{1,p}(\Omega)\setminus\{0\}}\frac{\|\nabla u\|^p_{L^{p}(\Omega)}}{\|u\|^p_{L^{p^*}(\Omega)}},$$
where $\Omega\subset\R^N$ is either a bounded set or $\R^N$.
\begin{proposition}\label{prop-boundedness}
Let $\Omega\subset\R^N$ be either a bounded open set or $\R^N$, $p\in(1,N)$ and $q\in(p,p^*)$.
\begin{itemize}
\item Let $u\ge 0$ be a subsolution to
\begin{equation}\label{diff-ineq}
-\Delta_p u\le u^{q-1}\qquad\text{in}\,\Omega,\qquad u\in W^{1,p}(\Omega).
\end{equation}
Then $u\in L^{\bar{q}}(\R^N)$ for any $\bar{q}\in[p^*,\infty]$ and 
\begin{equation}\label{est-u-Lq}
\|u\|_{L^{\bar{q}}(\Omega)}\le aS_p(\Omega)^{-c}\max\{\|u\|_{p^*},\,\|u\|_{p^*}^{1+\frac{q-p}{p^*-q}}\},
\end{equation}
for some constants $a=a(p,q)>0$ and $c=c(p,q)>0$ which do not depend neither on $\bar{q}$ nor on $\Omega$.
\item If, in addition, $u\in W^{1,p}_0(\Omega)$, the constant $S_p(\Omega)$ can be replaced by $S_p=S_p(\R^N)$ in \eqref{est-u-Lq}.
\end{itemize}

\end{proposition}
\begin{proof}
The proof follows the outlines of the proof of Proposition $2.4$ of \cite{PR}.\\

First we set, for $M>0$ and $x\in\Omega$, $v_M(x):=\min\{u(x),M\}$ and we test relation (\ref{diff-ineq}) with $v:=v_M^{p\chi+1}$, where $\chi>0$ will be fixed below. A direct computation shows that
$$\int_\Omega (p\chi+1)v_M^{p\chi}|\nabla v_M|^p dx
\le\int_\Omega u^{q-1}v dx.$$ 
Using the Sobolev embedding $W^{1,p}(\Omega)\subset L^{p^*}(\Omega)$, we can see that
\begin{equation}\notag
\begin{aligned}
&\int_\Omega u^{q-1}v dx\ge \int_\Omega (p\chi+1)v_M^{p\chi}|\nabla v_M|^p dx\\
&=\frac{p\chi+1}{(\chi+1)^p}\int_\Omega |\nabla (v_M^{\chi+1})|^p dx\ge S_p(\Omega)\frac{p\chi+1}{(\chi+1)^p}\|v_M\|^{p(\chi+1)}_{p^*(\chi+1)},
\end{aligned}
\end{equation}
where $S_p(\Omega)^{-1/p}$ is the best constant in the embedding $D^{1,p}(\Omega)\subset L^{p^*}(\Omega)$.\\ 

Note that, if $u\in W^{1,p}_0(\Omega)\subset W^{1,p}(\R^N)$, the constant $S_p(\Omega)$ can be replaced by $S_p=S_p(\R^N)$. For this reason the constant $C$ appearing in \ref{est-u-Lq} can be taken to be independent of $\Omega$ if $u\in W^{1,p}_0(\Omega)$.\\

On the other hand, by the definition of $v_M$ and the H\"older inequality, one has
\begin{equation}\notag
\int_\Omega u^{q-1}v dx\le\int_\Omega u^{q-p}u^{p(\chi+1)} dx\le \|u\|_{p^*}^{q-p}\|u\|^{p(\chi+1)}_{\gamma_0(\chi+1)},
\end{equation}
where we have set $\gamma_0:=\frac{p p^*}{p^*+p-q}\in(p,p^*)$ since $q\in(p,p^*)$. As a consequence, we have the estimate
$$\|v_M\|_{p^*(\chi+1)}\le \left(\frac{\chi+1}{(S_p(\Omega)(p\chi+1))^{1/p}}\right)^{\frac{1}{\chi+1}}\|u\|_{p^*}^{\frac{q-p}{p(\chi+1)}}\|u\|_{\gamma_0(\chi+1)},\qquad\forall\,M>0,\,\chi>0.$$
Using the fact that $u\ge 0$, $v_M\to u$ as $M\to\infty$ pointwise in $\Omega$ and the Fatou lemma,  we have as $M\to\infty$,
\begin{equation}\label{est-u-chi}
\|u\|_{p^*(\chi+1)}\le \left(\frac{\chi+1}{(S_p(\Omega)(p\chi+1))^{1/p}}\right)^{\frac{1}{\chi+1}}\|u\|_{p^*}^{\frac{q-p}{p(\chi+1)}}\|u\|_{\gamma_0(\chi+1)},\qquad\forall\,\chi>0. \end{equation}
Applying (\ref{est-u-chi}) with $\chi=\chi_1>0$ such that $1+\chi_1=p^*/\gamma_0$ we have
$$\|u\|_{p^*(\chi_1+1)}\le \left(\frac{\chi_1+1}{(S_p(\Omega)(p\chi_1+1))^{1/p}}\right)^{\frac{1}{\chi_1+1}}\|u\|_{p^*}^{1+\frac{q-p}{p(\chi_1+1)}}.$$
By induction, for any $n>0$ we can choose $\chi_n>0$ such that $1+\chi_n=(p^*/\gamma_0)^n$, so that, by \eqref{est-u-chi}
$$\|u\|_{p^*(\chi_n+1)}\le S_p(\Omega)^{-\frac{1}{p}\sum_{k=1}^n\frac{1}{\chi_k+1}}\prod_{k=1}^n\left(\frac{\chi_k+1}{(p\chi_k+1)^{1/p}}\right)^{\frac{1}{\chi_k+1}}\|u\|_{p^*}^{1+\frac{q-p}{p}\left(\sum_{k=1}^n\frac{1}{\chi_k+1}\right)},$$
for any $n\ge 1$. We note that $\chi_n\to\infty$. Using that the function $\varphi(y):=\left(\frac{y+1}{(py+1)^{1/p}}\right)^{\frac{1}{(y+1)^{1/p}}}$ is bounded on $(0,\infty)$ and the properties of the geometric series with general ratio $\gamma_0/p^*$, we conclude that
\begin{equation}
\begin{aligned}
\|u\|_{p^*(\chi_n+1)}&\le (\sup_{y\in(0,\infty)}\varphi(y))^{\sum_{k=1}^n\left(\frac{p^*}{\gamma_0}\right)^{k\left(1-\frac{1}{p}\right)}}S_p(\Omega)^{-c}\|u\|_{p^*}^{1+\frac{q-p}{p^*-q}(1-(\gamma_0/p^*)^n)}\\
&\le a S_p(\Omega)^{-c}\|u\|_{p^*}^{1+\frac{q-p}{p^*-q}(1-(\gamma_0/p^*)^n)},\qquad\forall\, n\ge 1,   
\end{aligned}   
\end{equation}
for some constants $a=a(p,q)>0,\,c=c(p,q)>0$ depending on $p$ but not on $\Omega$. We refer to the proof of \cite[Theorem 1.3]{PWZ} for the details.\\
								
Taking $\bar{q}>p^*$, $n\ge 1$ such that $p^*(\chi_n+1)>\bar{q}$ and $\alpha\in(0,1)$ defined by $$\frac{1}{\bar{q}}=\frac{\alpha}{p^*}+\frac{1-\alpha}{p^*(\chi_n+1)},$$  
an interpolation inequality gives
\begin{equation}
\begin{aligned}
\|u\|_{\bar{q}}&\le \|u\|_{p^*}^\alpha\|u\|^{1-\alpha}_{p^{*}(\chi_n+1)} \le a S_p(\Omega)^{-c}
\|u\|_{p^*}^{1+\frac{q-p}{p^*-p}(1-(\gamma_0/p^*)^n)(1-\alpha)}\\
&\le a S_p(\Omega)^{-c}\max\{\|u\|_{p^*},\|u\|_{p^*}^{1+\frac{q-p}{p^*-p}}\}=:S_p(\Omega)^{-c}g(\|u\|_{p^*}),    
\end{aligned} 
\end{equation}
where $a>0$ and $c>0$ are independent of $p,\,q$ and $\Omega$. This concludes the proof for $\bar{q}\in(p^*,\infty)$.\\
								
In orded to treat the case $q=\infty$, we note that the function $v(x):=\min\{u(x),S_p(\Omega)^{-c}g(\|u\|_{p^*})+1\}$ is bounded in $\Omega$ and satisfies $\|v\|_{\bar{q}}\le\|u\|_{\bar{q}}\le S_p(\Omega)^{-c}g(\|u\|_{p^*})$ for any $\bar{q}\ge p^*$, hence
$$\|v\|_\infty=\lim_{\bar{q}\to\infty}\|v\|_{\bar{q}}\le S_p(\Omega)^{-c}g(\|u\|_{p^*}).$$
This yields that $v(x)=\min\{u(x),S_p(\Omega)^{-c}g(\|u\|_{p^*})+1\}\le S_p(\Omega)^{-c}g(\|u\|_{p^*})$, hence $v=u$, so in particular $u\in L^\infty(\Omega)$ and $\|u\|_\infty\le S_p(\Omega)^{-c}g(\|u\|_{p^*})$.
\end{proof}
\begin{remark}
We stress that, if $\Omega\subset\R^N$ is a bounded open set and $u\in W^{1,p}(\Omega)$, the radially decreasing rearrangement $u^*$ of $u$ satisfies P\'olya-Szeg\H{o} inequality
$$\|\nabla u^*\|_{L^p(B_{|\Omega|})}\le\|\nabla u\|_{L^p(\Omega)},\quad\|u^*\|_{L^q(B_{|\Omega|})}=\|u\|_{{L^q(\Omega)}}$$
for any $q\in[p,p^*]$, where $B_{|\Omega|}$ is the ball of measure $|\Omega|$ centered at $0$. As a consequence, $S_p(\Omega)\le S_p(B_{|\Omega|})$. Therefore, the estimate provided by \eqref{est-u-Lq} can be rephrased by saying that
\begin{equation}\label{est-u-Lq-unif}
\|u\|_{L^{\bar{q}}(\Omega)}\le a\max\{\|u\|_{p^*},\,\|u\|_{p^*}^{1+\frac{q-p}{p^*-q}}\},
\end{equation}
for some constant $a>0$ depending on $p,\,q$ and $|\Omega|$ if $\Omega$ is bounded. If $u\in W^{1,p}_0(\Omega)$, the constant $a$ can be chosen to be independent of $\Omega$.
\end{remark}
Proposition \ref{prop-boundedness} is a version of the boundedness results given by Corollary $1.1$ of \cite{guedda1989quasilinear} for $q\in(p,p^*)$ in a wider class of domains (in \cite{guedda1989quasilinear} only bounded domains are considered). In dimension $N=3$ and for $p=2$, similar results in $\Omega=\R^3$ can be found in Theorem $1.3$ of \cite{pucci2024bifurcation} and for bounded domains in Proposition $2.4$ of \cite{PR}. However, the technique used in the proof of Proposition $2.4$ of \cite{PR} can be extended to possibly unbounded domains and any $p\in(1,N)$.\\

Now we are ready to prove Theorem \ref{cor-V>0}.
\begin{proof}[Proof of Theorem \ref{cor-V>0}]
Due to the assumptions about $p,\,q,\,s$ and $V$, we can see that $\Delta_p u\in L^1_{loc}(\R^N)$. Moreover, since $\nabla u\in L^p(\R^N)$, $|\nabla u|^{p-1}\in L^1(\partial B)$ for any ball $B\subset\R^N$. Therefore, by the Kato inequality for the $p$-Laplacian, stated in Theorem $3.4$ of~\cite{HORIUCHI2021}, we have $$\Delta_p u^+\ge \chi_{\{u\ge 0\}}\Delta_p u\qquad\text{in}\,\mathcal{D}'(\Omega),\,\forall\,\Omega\subset\subset\R^N.$$
As a consequence, $u^+\in W^{1,p}(\R^N)$ is a non-negative subsolution to $$-\Delta_p u^+\le f(u^+)-V(x)(u^+)^{p-1}\qquad\text{in}\,\R^N.$$ 
Since $f$ satisfies \eqref{growth-f-qs} and $V\ge 0$, $u^+$ is also a subsolution to
$$-\Delta_p u^+\le c(u^+)^{q-1}\qquad\text{in}\,\R^N.$$
As a consequence, by Proposition \ref{prop-boundedness}, $u^+\in L^\infty(\R^N)$. Since $f$ is odd, $-u$ is also a solution to \eqref{eq-pohozaev-intr}, therefore the same argument applies to $(-u)^+=u^-$. Hence $u^-\in L^\infty(\R^N)$ too. By Theorem \ref{cor-loc-bd}, this concludes the proof.
\end{proof}

\section{A deformation Theorem}\label{sec-defo}
In this section, we show how to construct a Palais-Smale sequence of a given functional (along {\em paths})  on a Hilbert manifold at a certain given inf-max level.\\

The setting is the following. We consider a Banach space $X$ 
Moreover, given a functional $G\in C^1(X,\R)$, we consider the set
$$M:=\{u\in X:\,G(u)=0\}\subset X.$$
We assume that $G'(u)\ne 0$ for any $u\in M$, so that $M$ is a $C^1$ manifold of codimension $1$.\\


Given a functional $J\in C^1(X,\R)$, we want to find a Palais-Smale sequence of $J$ (along {\em paths}) constrained on the manifold $M$.\\

For any $u\in M$, the tangent space to $M$ at $u$ is given by $$T_u M=\{v\in X:\,G'(u)[v]=0\}={\rm Ker}(G'(u))\subset X.$$
In these notations, for any $u\in M$, $J'_M$ is in the topological dual space $(T_u M)^\ast$ of $T_u M$. Then norm of a functional $F\in (T_u M)^\ast$ is given by
$$\|F\|_\ast:=\sup\{|F[v]|:\,v\in T_u M,\,\|v\|_{X}=1\}.$$
Let $T(U):=\cup_{u\in M}T_u M$ denote the tangent bundle to $M$ and let $\tilde{M}:=\{u\in M:\,J'_M(u)\ne 0\}$.
\begin{lemma}[Existence of a pseudo gradient]
\label{lemma-pseudo-grad}
In the above notations, there exists a vector field $Y:\tilde{M}\to T(U)$ such that
\begin{enumerate}
\item $\|Y(u)\|_X\le 2\|J'_M(u)\|_\ast,\,\,\forall\,u\in \tilde{M}$.
\item $J'_M(u)[Y(u)]\ge\|J'_M(u)\|^2_\ast,\,\forall\,u\in \tilde{M}$.
\end{enumerate}
\end{lemma}
\begin{proof}
For any $u\in M$, we take a tangent vector $w\in T_u M$ such that $$\|w\|_X=1,\qquad J'_M(u)[w]\ge\frac{3}{4}\|J'_M(u)\|_\ast$$ 
and we set $Y(u):=\frac{3}{2}\|J'_M(u)\|_\ast w\in T_u M$. It is possible to see that $Y(u)$ fulfills the required properties.  
\end{proof}
The vector field $Y$ constructed in Lemma \ref{lemma-pseudo-grad} is known as a \textit{pseudo-gradient} vector field for $J$ on $M$.
\begin{remark}
We note Lemma \ref{lemma-pseudo-grad} is a generalization of Lemma $4$ of \cite{lions1983nonlinear} where the authors prove the existence of a pseudo-gradient vector field in the particular case $S:=\{u\in H:\,\|u\|_H=1\}$.
\end{remark}
The main purpose of the Section is to find a sequence $\{u_n\}\subset M$ such that $J(u_n)$ converges to some suitable real number and $\|J'_M(u_n)\|_\ast\to 0$.\\

For our purposes, we need to introduce the notion of distance between a point $u\in M$ and a given set $C\subset X$. This can be done by setting $${\rm dist}(u,C):=\inf_{v\in C}\|u-v\|_X.$$
Now we are ready to state and prove the main result of the sequence, which gives the existence of a sequence with the properties mentioned above as a corollary.
\begin{theorem}\label{th-bd-PS}
Let $k\ge 1$ and let $K\subset\R^k$ be a compact subset. In the above notations, let $\Gamma\subset C^0(K,M)$. Assume that
\begin{equation}\label{MP-geom-J}
m:=\sup_{\gamma\in\Gamma}\max_{t\in \partial K} J(\gamma(t))<c:=\inf_{\gamma\in\Gamma}\max_{t\in K}J(\gamma(t)).  
\end{equation}
Moreover, assume that there exists $a\in(m,c)$ such that $\Gamma$ is invariant under homotopies that fix the sublevel set $J^a:=\{u\in M:\,J(u)<a\}$. Then, for any $\varepsilon>0$ and for any $\gamma\in\Gamma$ such that 
\begin{equation}\label{min-seq}
\max_{t\in K}J(\gamma(t))<c+\frac{\varepsilon}{2},   
\end{equation}  
there exists $u\in M$ such that 
\begin{enumerate}
    \item $|J(u)-c|<\varepsilon$,
    \item $\|J'_M(u)\|_\ast<\sqrt{\varepsilon}$
    \item ${\rm dist}(u,\gamma(K))<\sqrt{\varepsilon}$.
\end{enumerate}
\end{theorem}
\begin{remark}
A related result was proved in Theorem $4,5$ of \cite{ghoussoub1993duality}. However, some remarks are due
\begin{itemize}
\item the result in \cite{ghoussoub1993duality} is stated for Hilbert a manifold only, while here we have a result for a more general Banach manifold.
\item Since we assume that $\Gamma$ is invariant under homotopies which fix a certain sublevel set, not only a certain compact set, our result is much more flwxible for the applications. Note that, for example, in \cite{bartsch2021normalized}, the authors need to start from a known solution to some limit problem, while here this technical obstacle is overcome. In other words, we can construct solutions to an equation with a potential $V$, which may be trivial, without using the existence of a solution for $V=0$. Therefore, Theorem \ref{th-bd-PS} works both for the case $V\equiv 0$ and for more a general potential.
\end{itemize}
\end{remark}
\begin{proof}[Proof of Theorem \ref{th-bd-PS}]
Assume by contradiction that the statement is not true, or equivalently there exists $\varepsilon>0$ and $\gamma\in\Gamma$ such that $\max_{t\in K}J(\gamma(t))<c+\frac{\varepsilon}{2}$ and, for any $u$ in the set
$$\Lambda_\varepsilon:=\{u\in M:\, |J(u)-c|<\varepsilon,\,{\rm dist}(u,\gamma(K))<\sqrt{\varepsilon}\},$$
we have $\|J'_M(u)\|_\ast>\sqrt{\varepsilon}$.\\

Without loss of generality, we can assume that $c-\varepsilon>a$.\\

Let $g_\varepsilon:M\to[0,1]$ be a smooth function such that 
$$g_\varepsilon(u)=
\begin{cases}
1\qquad\forall\, u\in\Lambda_{\varepsilon/2} \\
0\qquad\forall\, u\in M:\,|J(u)-c|\ge \varepsilon.
\end{cases}
$$
and let
$$W(u):=
\begin{cases}
-g(u)\frac{Y(u)}{\|Y(u)\|_\ast}\qquad\forall\,u\in \tilde{M}\\
0\qquad\forall\,u\in M\setminus\tilde{M}.
\end{cases}
$$
Then $W(u)\in T_u M$, for any $u\in M$.\\

It is known that, given $u\in M$, the associated flow $\eta$ obtained by solving the Cauchy problem
\begin{equation}\notag
\begin{aligned}
\frac{d}{ds}\eta(s,u)&=W(\eta(s,u))\\   
\eta(0,u)&=u,
\end{aligned}    
\end{equation}
is defined for any $s\in[0,\infty)$. Moreover, $J$ is decreasing along the flow, since $$\frac{d}{ds}J(\eta(s,u))=\langle J'_M(\eta(s,u),W(\eta(s,u))\rangle\le-\|J'_M(\eta(s,u))\|_\ast\le 0\qquad\forall\, u\in M,\,s\ge 0.$$
In particular, $$\frac{d}{ds}J(\eta(s,u))\le-\sqrt{\varepsilon}\qquad\forall\,u\in\Lambda_\varepsilon,\,s\ge 0.$$ 
Moreover, for any $u\in M$ and $s\ge 0$, we have
\begin{equation}\notag
\|\eta(s,u)-u\|_X=\|\eta(s,u)-\eta(0,u)\|_X  =\left\|\int_0^s\frac{d}{d\tau}\eta(\tau,u)d\tau\right\|_X=\left\|\int_0^s W(\tau,u)d\tau\right\|_X\le \int_0^s \|W(\tau,u)\|_Xd\tau\le s.
\end{equation}
This yields that, for any $u\in\gamma(K)$ and $s\ge 0$, we have ${\rm dist}(\eta(s,u),\gamma(K))\le s$, so in particular ${\rm dist}(\eta(s,u),\gamma(K))<\sqrt{\varepsilon}$ if $s<\sqrt{\varepsilon}$.\\

As a consequence, taking $s_0:=\frac{3}{4}\sqrt{\varepsilon}>0$ and $u\in\gamma(K)$ such that $|J(u)-c|<\varepsilon$, we have
\begin{equation}
\label{decrease-J}
J(\eta(s_0,u))=J(u)+\int_0^{s_0}\frac{d}{ds}J(\eta(s,u)ds\le J(u)-s_0\sqrt{\varepsilon}<c+\frac{\varepsilon}{2}-s_0\sqrt{\varepsilon}=c-\frac{\varepsilon}{4}<c.    
\end{equation}
Therefore, setting $\tilde{\gamma}(t):=\eta(s_0,\gamma(t))$, we have $\max_{t\in K} J(\gamma(t))<c-\frac{\varepsilon}{4}$.\\

On the other hand, we have $\tilde{\gamma}(t)=\gamma(t)$ if $t\in K$ is such that $J(\gamma(t))\le c-\varepsilon$, so in particular
$$\max_{t\in\partial K}J(\tilde{\gamma}(t))=\max_{t\in \partial K}J(\gamma(t))\le m.$$
Furthermore, due to the choice of $g$ and $\varepsilon$, the function 
\begin{equation}\notag
H:[0,s_0]\times M\to M,\qquad H(s,u):=\eta(s,u)  
\end{equation}
is a homotopy which leaves the sublevel set $J^a$ invariant. Hence this homotopy leaves $\Gamma$ invariant, so that $\tilde{\gamma}(\cdotp)=H(s_0,\gamma(\cdotp))\in \Gamma$, a contradiction.
\end{proof}
\section{The right inverse of a non-linear operator}\label{sec-right-inv}
Given $N\ge 2$, $p\in(1,N)$, $s\in(1,p]$, 
$\lambda>0$ and $f\in X^*$, we want to solve the problem
\begin{equation}\label{quasi-lin-pb}
\begin{aligned}
&-\Delta_p u+|u|^{p-2}u+\lambda |u|^{s-2}u=f\qquad\text{in}\,X^* ,\\
& u\in X,
\end{aligned}    
\end{equation}
where we have set $X:=W^{1,p}(\R^N)\cap L^s(\R^N)$. Problem \eqref{quasi-lin-pb} has a variational structure. In fact, the solutions are the critical points of the functional 
$$E_f(u):=\frac{1}{p}\|u\|^p+\frac{\lambda}{s}\|u\|^s_s-\langle f,u\rangle\qquad\forall\,u\in X.$$
The space $X$ is endowed with the norm $$\|u\|_X:=\|u\|+\|u\|_s,\qquad\forall\, u\in X,$$
where we have set 
$$\|u\|:=(\|\nabla u\|_p^p+{\rm sgn}(p-s)\|u\|_p^p)^{1/p}=
\begin{cases}
(\|\nabla u\|^p_p+\|u\|_p^p)^{1/p}\qquad\text{if}\,p\in(1,s)\\
\|\nabla u\|_p\qquad\text{if}\,p=s
\end{cases}.$$
\begin{proposition}\label{prop-quasi-lin-pb}
Let $N\ge 2$, $p\in(1,N)$, $s\in(1,p]$, 
$\lambda>0$ and $f\in X^*$. Then
\begin{enumerate}
\item there exists a unique solution $u:=\Psi_\lambda(f)\in X$ to Problem \eqref{quasi-lin-pb}. Moreover, $u$ satisfies $E_f(u)=\min_{v\in X} E_f(v)$.
\item There exists a constant $C=C(p,q)>0$ such that 
\begin{equation}\label{est-inv}
\|u\|^p+\lambda\|u\|^s_s\le C (\|f\|_{X^*}^{p'}+\|f\|_{X^*}^2)\qquad\forall\,f\in X^*.
\end{equation}   
\item The inverse $\Psi_\lambda:X^*\to X$ is continuous with respect to the norms $\|\cdotp\|_{X^*}$ and $\|\cdotp\|_X$.
\end{enumerate}
\end{proposition}
\begin{proof}
\begin{enumerate} 

\item It is possible to see that the functional $E_f$ is coercive on $X$, since, for any $\varepsilon>0$, we have
\begin{equation}\notag
\begin{aligned}
E_f(u)&\ge \frac{1}{p}\|u\|^p+\frac{\lambda}{s}\|u\|^s_s-\|u\|_X\|f\|_{X^*}  \\
&\stackrel{\text{Young's}}{\ge}\frac{1}{p}\|u\|^p+\frac{\lambda}{s}\|u\|^s_s-\frac{\varepsilon^p}{p}\|u\|^p-\frac{1}{\varepsilon^{p'}p'}\|f\|^{p'}_{X^*}-\frac{\varepsilon^s}{s}\|u\|^s_s-\frac{1}{s'\varepsilon^{s'}}\|f\|^{s'}_{X^*}\\
&\ge\frac{1-\varepsilon^p}{p}\|u\|^p+\frac{\lambda-\varepsilon^s}{s}\|u\|^s_s-\frac{1}{\varepsilon^{p'}p'}\|f\|^{p'}_{X^*}-\frac{1}{s'\varepsilon^{s'}}\|f\|^{s'}_{X^*},
\end{aligned}   
\end{equation}
Therefore, coercivity can be proved by choosing $\varepsilon>0$ small enough.\\ 

Moreover, the functional $E_f$ is continuous with respect to the norm $\|\cdotp\|_X$ and strictly convex. Hence it follows that it has a unique critical point, which coincides with the global minimizer.

\item Testing the equation with $u$ we can see that, for any $\varepsilon>0$, we hvae
\begin{equation}\notag
\|u\|^p+\lambda\|u\|_s^s=\langle f,u\rangle\le\|u\|_X\|f\|_{X^*}
\stackrel{\text{Young's}}{\le} \frac{\varepsilon^p}{p}\|u\|^p+\frac{1}{\varepsilon^{p'}p'}\|f\|_{X*}^{p'}+\frac{\varepsilon^s}{s}\|u\|^s_s+\frac{1}{s'\varepsilon^{s'}}\|f\|_{X*}^{s'}
\end{equation}
Taking $\varepsilon>0$ so small enough, estimate \eqref{est-inv} is true.

\item In order to show continuity, we take a sequence $\{f_k\}_{k}\subset X^*$ such that $f_k\to f$ in $X^*$ and we show that $u_k:=\Psi_\lambda(f_k)\to u:=\Psi_\lambda(f)$ in $X$. In order to do so, we show that any subsequence $\{f_{k_j}\}_j$ admits a subsubsequence such that $u_{k_{j_i}}\to u$.\\

Let us consider a subsequence, still denoted by $\{f_k\}_k$. We note that $f_k$ is bounded in $X^*$, since $f_k\to f$ in $X^*$. As a consequence, estimate \eqref{est-inv} yields that $u_k$ is bounded in $X$. Therefore, by the weak compactness of the ball in the reflexive Banach spaces $W^{1,p}(\R^N)$ and $L^s(\R^N)$, we can extract a subsequence, namely $\{u_{k_j}\}_j$, such that $u_{k_j} \rightharpoonup v$ weakly in $X$, for some $v\in X$.\\

The weak convergence yields that $v$ also solves problem \eqref{quasi-lin-pb}, hence $v=u$, due to uniqueness.  
Since every subsequence of $\{u_k\}$ has a further subsequence weakly converging to $u$, the entire sequence satisfies $u_k \rightharpoonup u$ weakly in $X$.\\

It remains to prove that $u_k\to u$ strongly in $X$. Testing the equation satisfied by $u_k$ with $u_k$ and the one satisfied by $u$ with $u$ we get
\begin{equation}
\begin{aligned}
\|u_k\|^p+\lambda\|u_k\|^s_s&=\langle f_k ,u_k\rangle    \\
\|u\|^p+\lambda\|u\|^s_s&=\langle f,u\rangle  
\end{aligned}    
\end{equation}
Taking the difference, we have
$$
\|u_k\|^p-\|u\|^p+\lambda(\|u_k\|^s_s-\|u\|^s_s)=\langle f_k, u_k\rangle-\langle f, u\rangle =o_k(1),$$
since $$\langle f_k,u_k\rangle-\langle f,u\rangle=\langle(f_k-f),u_k\rangle+\langle f,(u_k-u)\rangle=o_k(1),$$
due to the fact that $u_k\rightharpoonup u$ weakly in $X$ and $f_k\to f$ strongly in $X^*$.\\

Moreover, the Brezis-Lieb Lemma yields that
\begin{equation}
\begin{aligned}
\|\nabla u_k\|^p_p-\|\nabla u\|^p_p&=\|\nabla(u_k-u)\|^p_p+o_k(1) \\
\|u_k\|^t_t-\|u\|^t_t&=\|u_k-u\|^t_t+o_k(1),\qquad\forall\,t\in\{s\}\cup[p,p^*]
\end{aligned}
\end{equation}
as a consequence $$\|u_k-u\|^p+\lambda\|u_k-u\|^s_s=o_k(1).$$
This concludes the proof.
\end{enumerate}
\end{proof}
\section{The proof of Theorem \ref{main-th-hom}}\label{sec-proof}
Note that the weak solutions to equation
\begin{equation}\notag
-\Delta_p u+|u|^{p-2}u+V(x)|u|^{p-2}u+\lambda |u|^{s-2}u=|u|^{q-2}u\qquad\text{in $\R^N$}    
\end{equation}
are the critical points of the functional
$$J_{V,\lambda}(u):=J_V(u)+\frac{\lambda}{s}\|u\|^s_s$$
on $X:=W^{1,p}(\R^N)\cap L^s(\R^N)$. However, due to the radial symmetry of $V$, it is natural to expect to be able to construct a radial solution $u\in X_r:=W^{1,p}_r(\R^N)\cap L^s(\R^N)\subset X $. In order to do so, we restrict ourselves to finding a critical point of $J_{V,\lambda}|_{X_r}$. This is enough provided that we show that the critical points of $J_{V,\lambda}|_{X_r}$ are actually true critical points of $J_{V,\lambda}$. This is known to be true, for example, in case $p=s=2$, since $X=H^1(\R^N)$ is a Hilbert space, hence Theorem $2.2$ of \cite{kobayashi2004principle} holds. However, in the general case that Theorem does not necessary apply and some remarks are due.

\subsection{Symmetry}
However, due to the symmetric criticality principle, namely Proposition $2.1$ of \cite{kobayashi2004principle}, it is possible to prove that the critical points of $J_{V,\lambda}|_{X_r}$ are actually true critical points of $J_{V,\lambda}$. \\

More precisely, the setting of \cite{kobayashi2004principle} is the following. Assume that $B$ is a Banch space, $G$ is a group, $J:B\to\R$ is a $C^1$-functional and $\pi: G\to {\rm Gl}(B)$ is a representation of $G$ over $B$, that is, for any $g\in G$, $\pi(g)$ is a bounded linear operator on $B$ and
\begin{itemize}
    \item $\pi(e)u=u,\,\forall\,u\in B$
    \item $\pi(g_1 g_2)u=\pi(g_1)(\pi(g_2)u),\,\forall\,u\in B$,
\end{itemize}
where $e\in G$ is the neutral element. Assume furthermore that $J$ is $G$-invariant, that is $$J(\pi(g)u)=J(u)\qquad\forall\,u\in B,\,g\in G.$$ 
Let 
\begin{equation}\label{def:sym:1}
\begin{aligned}
\Sigma&:=\{u\in B:\,\pi(g)u=u,\,\forall\,g\in G\},\\
\Sigma_*&:=\{v^*\in B^*:\,\langle v^*,u-\pi(g)u\rangle=0,\,\forall\,g\in G\}.\\
\Sigma^\bot&:=\{v^*\in B^*:\,\langle v^*,u\rangle=0,\,\forall\,u\in\Sigma\}.
\end{aligned}
\end{equation}
In these notations, the following statement holds.
\begin{proposition}[\cite{kobayashi2004principle}]\label{prop-scp}
If $\Sigma_*\cap\Sigma^\bot=\{0\}$ and $u\in B$ satisfies $\langle J'(u),v\rangle=0$ for any $v\in \Sigma$ , then $u\in \Sigma$ and $u$ is a critical point of $J$.  
\end{proposition}
In our case, we apply Proposition \ref{prop-scp} with $B:=W^{1,p}(\R^N)\cap L^s(\R^N)$, $G:=O(N)$ and $\Sigma:=X_r=\{u\in B:\,u=u_r\}$, where
$$u_r(x):=\int_{O(N)} u\circ g(x)\,d\mu_g\qquad\forall\, u\in B,\,g\in G.$$
Note that the mapping $\phi:u\in X\mapsto u_r\in X_r$ is a projection, in the sense that $\phi^2=\phi$. In other words, for any $u\in X$ we have $(u_r)_r=u_r\in \Sigma$.\\ 

Furthermore, $\Sigma_*\cap\Sigma^\bot=\{0\}$. In fact, for $v^*\in\Sigma_*\cap\Sigma^\bot$, we have
\begin{equation}
\begin{aligned}
\langle v^*,u\rangle&=\langle v^*,u-u_r\rangle+\langle v^*,u_r\rangle=\langle v^*,u-u_r\rangle=\langle v^*,u-\int_{O(N)}(u\circ g) \,d\mu_g\rangle\\
&=\langle v^*,\int_{O(N)}(u-u\circ g) \,d\mu_g\rangle= \int_{O(N)}\langle v^*,(u-u\circ g)\rangle \,d\mu_g
\stackrel{\eqref{def:sym:1}}{=}0,\qquad\forall\, u\in B,
\end{aligned}
\end{equation}
which yields that $\Sigma_*\cap\Sigma^\bot=\{0\}$. As a consequence, applying the above argument to $J:=J_{V,\lambda}$, it is enough to look for critical points of $J_{V,\lambda}|_{X_r}$.

\subsection{Mountain pass geometry}
In this section we prove that $J_V$ has the mountain-pass geometry. The proof follows the outlines of Section $2.2$ of our first paper.\\

First let us recall the Gagliardo-Nirenberg inequality, which asserts that there exists a constant $C>0$ such that
\begin{equation}
\|u\|_q\le C\|\nabla u\|_p^\theta \|u\|_s^{1-\theta}   \qquad\forall\, u\in W^{1,p}(\R^N),\,\theta:=\frac{pN(q-s)}{q(p(N+s)-sN)}.
\end{equation}
We note that $\theta\in(0,1)$ since $s\le p<q<p^*$. \\

Let us denote, for $k>0$,
\begin{align}
\alpha_k:=\sup_{u\in D_k}J_V(u);\quad \beta_k:=\inf_{u\in \partial D_k}J_V(u),\\
D_k:= \{u\in \mathcal{S}_{\rho,r}:\  \|u\|\le k\},
\end{align}
where $\|\cdotp\|$ denotes the standard $W^{1,p}(\rn)$-norm.
\begin{lemma}
\label{lemma-MP}
Assume that either $V\in L^{\alpha}(\rn)$ with $\alpha\in[\frac{N}{p},\infty]$ and $\|V^-\|_{\alpha}<S_{p,\alpha}$. 
Then for $\frac{N+s}{N}p<q<p^*$, there exist $0<k_1<k_2$ such that
\begin{align}\label{eq:lev:1}
0<\alpha_{k_1}<\beta_{k_2}.
\end{align}
and $\sup_{u\in D_{k_1}}J_V(u)>0$.
\end{lemma}
\begin{proof}
\begin{itemize}
\item[i)] Let us consider the case $s\in(1,p)$ and $\alpha\in[N/p,\infty]$ or $s=p$ and $\alpha=N/p$.\\

If $V\in L^\alpha(\R^N)$ with $\alpha\in[\frac{N}{p},\infty]$, by the H\"{o}lder inequality and the Sobolev embedding $D^{1,p}(\rn)\subset L^{p^*}(\rn)$
, we have
\begin{equation}\notag
\left|\int_{\rn}V^-(x)|u|^pdx\right|\le\|V^-\|_{\alpha}\|u\|^p_{\frac{p\alpha}{\alpha-1}}\le S_{p,\alpha}^{-1}\|V^-\|_{\alpha}\|u\|^p\qquad\forall\,u\in \mathcal{S}_{\rho,r}.
\end{equation}
As a consequence, writing $V=V^+-V^-$, we have 
\begin{equation}\label{low-est-JV}
\begin{aligned}
J_V(u)&\ge \frac{1}{p}(1-S_{p,\alpha}^{-1}\|V^-\|_{\alpha})\|u\|^p
- \frac{1}{q}\|u\|_q^q\\
    &
    \ge\frac{1}{p}(1-S_{p,\alpha}^{-1}\|V^-\|_{\alpha})\|u\|^p
    - \frac{C^q}{q}\|\nabla u\|_p^{\theta q}\|u\|_s^{(1-\theta)q}\\
    &\ge \frac{1}{p}(1-S_{p,\alpha}^{-1}\|V^-\|_{\alpha})\|u\|^p-\frac{C^q}{q}\rho^{(1-\theta)q}\|u\|^{\theta q}=f(\|u\|),\qquad \forall\, u\in \mathcal{S}_{\rho},
\end{aligned}
\end{equation}
where we have set $$f(t):=\frac{1}{p}(1-S_{p,\alpha}^{-1}\|V^-\|_{\alpha})t^p-\frac{C^q}{q}\rho^{(1-\theta)q}t^{\theta q}.$$\\

We note that $f(0)=0$. Using that $1-S_{p,\alpha}^{-1}\|V^-\|_{\alpha}>0$, and $\theta q>p$, since $p\frac{N+s}{N}<q<p^*$, we can see that there exists a unique $t_\rho>0$ such that $$f(t_\rho)=\max_{t\in(0,\infty)}f(t)>0$$ 
$f$ is strictly increasing in $(0,t_\rho)$ and strictly decreasing in $(t_\rho,\infty)$.\\

As a consequence, (\ref{low-est-JV}) yields that, setting $k_2=t_\rho$, we have $\beta_{k_2}\ge f(t_\rho)>0$ and $J_V(u)>0$ for any $u\in D_{k_1}\setminus\{0\}$.\\

On the other hand, taking $k_1<\bar{k}_1:=\min\{(\frac{pf(t_\rho)}{1+\|V^+\|_\alpha S^{-1}_{p,\alpha}})^{\frac{1}{p}},k_2\}$, we have
\begin{equation}\label{low-est-J}
0\le J_V(u)\le \frac{1}{p}(1+\|V^+\|_\alpha S^{-1}_{p,\alpha})\|u\|^p\le\frac{k_1^p}{p}(1+\|V^+\|_\alpha S^{-1}_{p,\alpha})<f(t_\rho)\le\beta_{k_2}\qquad\forall\,u\in D_{k_1},    
\end{equation}
which concludes the proof.
\item Let us consider the case $\alpha\in(N/p,\infty]$ and $s=p$.\\

In this case, we have
\begin{equation}\notag
\left|\int_{\rn}V^-(x)|u|^pdx\right|\le\|V^-\|_{\alpha}\|u\|^p_{\frac{p\alpha}{\alpha-1}}\le S_{p,\alpha}^{-1}\|V^-\|_{\alpha}\|u\|^p_{W^{1,p}(\R^N)}\qquad\forall\,u\in \mathcal{S}_{\rho,r}.
\end{equation}
As a consequence, we can see that
\begin{equation}
\begin{aligned}
J_V(u)&\ge \frac{1}{p}(1-S_{p,\alpha}^{-1}\|V^-\|_\alpha)\|\nabla u\|^p_p-\frac{C^q}{q}\rho^{(1-\theta)q}t^{\theta q}-\frac{\rho^p}{p}\|V^-\|_\alpha S^{-1}_{p,\alpha}\\
&=f(\|\nabla u\|_p)-\frac{\rho^p}{p}\|V^-\|_\alpha S^{-1}_{p,\alpha}=:g(\|\nabla u\|_p).
\end{aligned}
\end{equation}
It is possible to see that, for any $\rho>0$, there exists a unique $t_\rho>0$ such that $$g(t_\rho)=\max_{t>0}g(t)\to\infty,\qquad\rho\to 0^+.$$
More precisely, it is possible to show that $t_\rho=O(\rho^{-\frac{q(1-\theta)}{\theta q-p}})\to \infty$ and $g(t_\rho)=O(\rho^{-\frac{pq(1-\theta)}{\theta q-p}})\to\infty$ as $\rho\to 0^+$.\\

As a consequence, setting $k_2=t_\rho$, we have $\beta_{k_2}\ge g(t_\rho)>0$ for $\rho>0$ small enough.\\

On the other hand, taking $1<k_1<\bar{k}_1:=\min\{(\frac{pg(t_\rho)}{1+\|V^+\|_\alpha S^{-1}_{p,\alpha}})^{\frac{1}{p}},k_2\}$, which is possible since 
$$\min\{(\frac{pg(t_\rho)}{1+\|V^+\|_\alpha S^{-1}_{p,\alpha}})^{\frac{1}{p}},k_2\}\to\infty\qquad\rho\to 0^+,$$
we have
\begin{equation}\label{low-est-J}
J_V(u)\le \frac{1}{p}(1+\|V^+\|_\alpha S^{-1}_{p,\alpha})\|u\|^p_{W^{1,p}}\le\frac{k_1^p}{p}(1+\|V^+\|_\alpha S^{-1}_{p,\alpha})<g(t_\rho)\le\beta_{k_2}\qquad\forall\,u\in D_{k_1},    
\end{equation}
which concludes the proof.
\end{itemize}
\end{proof}
Lemma \ref{lemma-MP} can be used to prove that $J_V$ has the mountain pass geometry. More precisely, we recall that 
we define the mountain-pass level
\begin{equation}
c^r_{V,\rho}:=\inf_{\gamma\in \Gamma_{\rho,r}}\max_{t\in[0,1]}J_V(\gamma(t)),
\end{equation}
where $\Gamma_{\rho,r}$ is the set of paths
$$\Gamma_{\rho,r}:=\{\gamma\in C([0,1],\mathcal{S}_{\rho,r}):\,\gamma(0)\in D_{k_1},\,J_V(\gamma(1))<0\}.$$
\begin{lemma}\label{lemma-mpg}
Assume that $V\in L^{\alpha}(\R^N)$ with $\alpha\in[\frac {N}{p},\infty]$ and $\|V^-\|_{\alpha}<S_{p,\alpha}$
. Then, in the above notations, we have
$$0<\sup_{\gamma\in\Gamma_{\rho,r}}\max\{J_V(\gamma(0)),\,J_V(\gamma(1))\}<c^r_{V,\rho}.$$
\end{lemma}
\begin{proof}
By definition of $\Gamma_{\rho,r}$, the definition of $\sup$, Lemma \ref{lemma-MP} and \eqref{low-est-JV}, there exists $u_0\in D_{k_1}$ such that $J_V(u_0)>0$. Taking a path $g\in \Gamma_{\rho,r}$ such that $g(0)=u_0$, we have $$\sup_{\gamma\in\Gamma_{\rho,r}}\max\{J_V(\gamma(0)),\,J_V(\gamma(1))\}\ge J_V(g(0))>0,$$
since Lemma \ref{lemma-MP} holds. Now we fix $\gamma\in\Gamma_{\rho,r}$. Since $J_V(\gamma(1))<0$, Lemma \ref{lemma-MP} yields that $\|\gamma(1)\|>k_2$. Being $\|\gamma(0)\|\le k_1<k_2$, by continuity there exists $\bar{t}\in(0,1)$ such that $\|\gamma(\bar{t})\|=k_2$, hence $$\max_{t\in[0,1]} J_V(\gamma(t))\ge J_V(\gamma(\bar{t}))\ge\beta_{k_2}>\alpha_{k_1}\ge J_V(\gamma(0))=\max\{J_V(\gamma(0)),J_V(\gamma(1))\}.$$  
Since $\gamma$ is arbitrary, this yields that $$c^r_{V,\rho}\ge \beta_{k_2}>\alpha_{k_1}\ge\sup_{\gamma\in\Gamma_{\rho,r}}\max\{J_V(\gamma(0)),J_V(\gamma(1))\}.$$
\end{proof}

\subsection{A bounded Palais-Smale sequence}
For $t>0$ and $u\in L^s(\R^N)$, we introduce the scaling 
$$u^t(x):=t^{\frac{N}{s}}u(tx).$$
Note that $\|u^t\|_s=\|u\|_s$, so that, if $u\in\mathcal{S}_\rho$, we have $u^t\in\mathcal{S}_\rho$ for any $t>0$. Using this scaling, we define the auxiliary functional
$$\tilde{J}_V:(u,\tau)\in X_r\times(0,\infty)\mapsto J_V(u^\tau)\in\R.$$
Applying Theorem \ref{th-bd-PS} to $\tilde{J}_V|_{\mathcal{S}_{\rho,r}\times(0,\infty)}$ we can construct a Palais-Smale sequence $\{u_n\}\subset \mathcal{S}_{\rho,r}$ of $J_V|_{\mathcal{S}_{\rho,r}}$ at level $c_{\rho}$ which almost satisfies the Pohozaev identity, in the sense that $P(u_n)=0$ (see \eqref{pohozaev-functional} for the definition of $P$).\\

More precisely, we have the following Lemma.
\begin{lemma}\label{lemma-bdPS}
There exists a sequence $\{u_n\}\subset \mathcal{S}_{\rho,r}$ such that
\begin{equation}\label{PS-pohozaev}
u_n^-\to 0 \,\text{in}\,X,\quad J_V(u_n)\to c^r_{V,\rho},\quad \|(J_V|_{\mathcal{S}_{\rho,r}})'(u_n)\|_*\to 0,\quad P(u_n)\to 0.
\end{equation}
\end{lemma}
\begin{proof}
Consider a sequence of paths $\{\gamma_n\}\subset\Gamma_{\rho,r}$ such that
\begin{equation}
\label{min-path}
\max_{t\in[0,1]}J_V(\gamma_n(t))<c^r_{V,\rho}+\frac{1}{n}.    
\end{equation}
Using that $J_V(|u|)=J(u)$, for any $u\in \mathcal{S}_{\rho,r}$, the sequence $\{|\gamma_n|\}\subset \Gamma_{\rho,r}$ still satisfies \eqref{min-path}, hence we can assume that $\gamma_n(t)\ge 0$ a. e. in $\R^N$, for any $t\in[0,1]$.\\

Set
$$d^r_{V,\rho}:=\inf_{\tilde{\gamma}\in\tilde{\Gamma}_{\rho,r}}\max_{t\in[0,1]}\tilde{J}_V(\tilde{\gamma}(t)),$$
where
\begin{equation}\notag
\begin{aligned}
\tilde{\Gamma}_{\rho,r}:&=\{\tilde{\gamma}:=(\gamma,\tau)\in C^0([0,1],\mathcal{S}_{\rho,r}\times(0,\infty)):\,\|\gamma(0)\|\le K_\rho,\,J_V(\gamma(1))<0,\,\tau(0)=\tau(1)=1\}\\
&=\Gamma_{\rho,r}\times\{\tau:[0,1]\to(0,\infty):\,\tau\,\text{is continuous and}\,\tau(0)=\tau(1)=1\}.     
\end{aligned}
\end{equation}
We note that $d^r_{V,\rho}=c^r_{V,\rho}$.\\

Note that $\tilde{\gamma}_n(t):=(\gamma_n(t),1)\in\tilde{\Gamma}_{\rho,r}$, where $\gamma_n$ is the sequence of paths considered in \eqref{min-path}, for any $\gamma_n\in \Gamma_{\rho,r}$, then we have
\begin{equation}
\label{min-path}
\max_{t\in[0,1]}\tilde{J}_V(\tilde{\gamma}_n(t))<d^r_{V,\rho}+\frac{1}{n}=c^r_{V,\rho}+\frac{1}{n}.    
\end{equation}

Note that $\tilde{\Gamma}_{\rho,r}$ satisfies the hypothesis of Theorem \ref{th-bd-PS}. Therefore, we obtain a sequence $(v_n,\tau_n)\in\mathcal{S}_{\rho,r}\times (0,\infty)$ such that
\begin{enumerate}
    \item $|\tilde{J}_V(v_n,\tau_n)-d^r_{V,\rho}|<\frac{2}{n}$
    \item $\|\partial_u\tilde{J}_V(v_n,\tau_n)\|_*+|\partial_\tau\tilde{J}_V(v_n,\tau_n)|<\sqrt{\frac{2}{n}}$
    \item $|\tau_n-1|+\max_{t\in[0,1]}\|v_n-(\tilde{\gamma}_1(t))_n\|_X<\sqrt{\frac{2}{n}}$.
\end{enumerate}
As a consequence, the functions $u_n:=v_n^{\tau_n}$ satisfy \eqref{PS-pohozaev}.
\end{proof}
Now we will show that, under suitable assumptions about the potential, the Palais-Smale sequence $\{u_n\}$ constructed in Lemma \ref{lemma-bdPS} is bounded.
\begin{lemma}\label{lemma-bd-lambda>0}
Assume that either $(V_1)$ 
is satisfied. 
Then 
\begin{enumerate}
    \item the Palais-Smale sequence $\{u_n\}\subset \mathcal{S}_{\rho,r}$ constructed in Lemma \ref{lemma-bdPS} is bounded in $W^{1,p}(\rn)$;
    \item there exists $\rho^*>0$ such that, for any $\rho\in(0,\rho^*)$, the sequence of Lagrange multipliers $\lambda_n:=\frac{\langle J'_V(u_n),u_n\rangle}{\rho^2}$ satisfies 
    \begin{equation}\label{lambda-infty}
    0<\lambda_\rho:=\liminf_{n\to\infty}\lambda_n\le \limsup_{n\to\infty}\lambda_n<\infty,\qquad\forall\,\rho\in(0,\rho^*).
    \end{equation}
    Moreover, $\rho^s\lambda_\rho\to \infty$ as $\rho\to 0^+$.
\end{enumerate}
\end{lemma}
\begin{proof}
Assume first that $(V_1)$ holds. Let $u_n$ be the PS-sequence  constructed in Lemma \ref{lemma-bdPS}. Set
\begin{equation}
\begin{aligned}
&A_n:=\|\nabla u_n\|_p^p,\,B_n:={\rm sgn}(p-s)\|u_n\|^p_p,\,C_n:=\int_{\rn}V(x)|u_n|^pdx,\\
&D_n:=\int_{\rn}V(x)|u_n|^{p-2}u_n\nabla u_n\cdotp x\, dx,\, E_n:=\|u_n\|_q^q.    
\end{aligned}
\end{equation}
We stress that, even if it is not explicitly indicated, $A_n,\, B_n,\, C_n,D_n$ and $E_n$ actually depend of $\rho$. \\

Then relations \eqref{PS-pohozaev} are equivalent to the system
\begin{subequations}
\begin{align}
&A_n+B_n+C_n-\frac{p}{q}E_n=p c_{\rho}+o_n(1) \label{syst-PS-seq-explicit:1}\\
&\frac{p(N+s)-sN}{p}A_n+\frac{N(p-s)}{p}B_n-\frac{N(q-s)}{q}E_n+N C_n+sD_n=o_n(1)\label{syst-PS-seq-explicit:2}\\
&A_n+B_n+C_n+\lambda_n\rho^s-E_n=o_n(1)(A_n+B_n+1)^{1/p}
\label{syst-PS-seq-explicit:3}
\end{align}
\label{syst-PS-seq-explicit}
\end{subequations} 
Subtracting \eqref{syst-PS-seq-explicit:2} from \eqref{syst-PS-seq-explicit:1} we can see that
$$\frac{sN-p(N+s-1)}{p}A_n+\frac{p-N(p-s)}{p}B_n+(1-N)C_n-sD_n+\frac{N(q-s)-p}{q}E_n=pc_{\rho}+o_n(1).$$
Multiplying~\eqref{syst-PS-seq-explicit:1} by $\frac{N(q-s)-p}{p}$ and adding to the above
\begin{equation}
\label{rel-bd}
\frac{Nq-p(N+s)}{p}A_n+\frac{N(q-p)}{p}B_n+\frac{N(q-p-s)}{p}C_n-sD_n=N(q-s)c_{\rho}+o_n(1).
\end{equation}
Note that $Nq-p(N+s)>0\Leftrightarrow q>\frac{p(N+s)}{N}$.\\
By the H\"{o}lder inequality and the Sobolev embedding $W^{1,p}(\rn)\subset L^{\frac{p\alpha}{\alpha-1}}(\rn)$ we have
\begin{equation}
\label{est-CD}
\begin{aligned}
&\left|\int_{\rn}V(x)|u|^p dx\right|\le \|V\|_{\alpha}\|u\|_{\frac{p\alpha}{\alpha-1}}^p\le S_{p,\alpha}^{-1}\|V\|_{\alpha}\|u\|^p,\\
&\left|\int_{\rn}V(x)|u|^{p-2}u\nabla u\cdotp x\right|\le\|\tilde{W}\|_{\frac{p\alpha}{p-1}}\|u\|_{\frac{p\alpha}{\alpha-1}}^{p-1}\|\nabla u\|_p
\le S_{p,\alpha}^{-\frac{p-1}{p}}\|\tilde{W}\|_{\frac{p\alpha}{p-1}}
\|u\|^{p-1}\|\nabla u\|_p\\
&\le S_{p,\alpha}^{-\frac{p-1}{p}}\|\tilde{W}\|_{\frac{p\alpha}{p-1}}\|u\|^p,  
\end{aligned}    
\end{equation}
Then the term $C_n$ and $D_n$ can be estimated as follows.
\begin{align*}
    sD_n - \frac{N(q-p-s)}{p}C_n
    \le\left( s S_{p,\alpha}^{-\frac{p-1}{p}}\|\tilde{W}\|_{c(\alpha)} 
    + \frac{N|q-p-s|}{p} S_{p,\alpha}^{-1}\|V\|_{\alpha}\right)(A_n+B_n),
\end{align*}
which gives us for an $n_0>0$ such that
\begin{align*}
    \left(\frac{Nq-p(N+s)}{p} -sS_{p,\alpha}^{-\frac{p-1}{p}}\|\tilde{W}\|_{c(\alpha)} 
    - \frac{N|q-p-s|}{p} S_{p,\alpha}^{-1}\|V\|_{\alpha} 
    \right)(A_n+B_n)\\
    \le 
      N(q-s)m_{V,\rho}+o_n(1)\quad\forall n\ge n_0,
\end{align*}
and $B_n\ge 0$.
Thanks to~\eqref{cond-norm-V}, we have
\[
\frac{Nq-p(N+s)}{p} > sS_{p,\alpha}^{-\frac{p-1}{p}}\|\tilde{W}\|_{\frac{p\alpha}{p-1}} 
    + \frac{N|q-p-s|}{p} S_{p,\alpha}^{-1}\|V\|_{\alpha}.
\]
This yields that $A_n+B_n$ is bounded.\\
Due to the Sobolev embeddings, $E_n$ is also bounded. Hence, due to \eqref{est-CD}, $C_n$ and $D_n$ are bounded as well. 
Therefore, from~\eqref{syst-PS-seq-explicit:3} we see that $\lambda_n$ is bounded too.\\

From~\eqref{syst-PS-seq-explicit:2}, we have, for a large $n_1\ge 0$,
\begin{equation}\notag
\begin{aligned}
\frac{p(N+s)-sN}{p}A_n
&\le\frac{p(N+s)-sN}{p}A_n+\frac{N(p-s)}{p}B_n\\
&=\frac{N(q-s)}{q}E_n-NC_n-sD_n+o_n(1)\quad\forall n\ge n_1.
\end{aligned}    
\end{equation}
Using the Gagliardo-Nirenberg inequality we have
$$E_n\le \kappa\|\nabla u_n\|_p^{\theta q}\|u_n\|_s^{(1-\theta)q}=\kappa A_n^{\theta\frac{q}{p}}\rho^{(1-\theta)q}.$$
Moreover, \eqref{est-CD} gives
$$|C_n|\le S_{p,\alpha}^{-1}\|V\|_{\alpha}A_n,\qquad |D_n|\le S_{p,\alpha}^{-\frac{p-1}{p}}\|\tilde{W}\|_{\frac{p\alpha}{p-1}}A_n.$$
As a consequence
$$\left(\frac{p(N+s)-sN}{p}-NS_{p,\alpha}^{-1}\|V\|_{\alpha}-sS_{p,\alpha}^{-\frac{p-1}{p}}\|\tilde{W}\|_{\frac{p\alpha}{p-1}}\right)A_n^{1-\theta\frac{q}{p}}
\le\tilde{\kappa}\rho^{(1-\theta)q}.$$
Using that $1-\theta\frac{q}{p}<0$ and
\begin{equation}
\label{extra-cond-V1}
\frac{p(N+s)-sN}{p}-N S_p^{-1}\|V\|_{\alpha}-s S_p^{-\frac{p-1}{p}}\|\tilde{W}\|_{\frac{p\alpha}{p-1}}>0,
\end{equation}
due to \eqref{cond-norm-V} and the fact that
$$s\left(\frac{N}{q}-\frac{N-p}{p}\right)<\frac{p(N+s)-sN}{p},$$
we deduce that there exists $\bar{\rho}>0$ and a continuous function $g:(0,\bar{\rho})\to\R$ such that 
\begin{equation}\label{A-n-bd-below}
\lim_{\rho\to 0^+}g(\rho)=\infty,\qquad A_n\ge g(\rho),\,\forall\,\rho\in(0,\bar{\rho}).
\end{equation}
In particular, $A_n\to\infty$ as $\rho\to 0^+$ uniformly in $n$.\\

Multiplying \eqref{syst-PS-seq-explicit:3} by $\frac{N(q-s)}{sq}$ and~\eqref{syst-PS-seq-explicit:2} by $1/s$ and subtracting, we can see that
\begin{equation}\label{rel-lambda-n}
\begin{aligned}
\lambda_n\rho^s\left(\frac{N}{s}-\frac{N}{q}\right)&=\left(\frac{N}{q}-\frac{N-p}{p}\right)A_n -\left(\frac{N}{p}-\frac{N}{q}\right)B_n +\left(N-\frac{N}{s}+\frac{N}{q}\right)C_n +D_n \\
&\ge\left(\frac{N}{q}-\frac{N-p}{p}-\left(N-\frac{N}{s}+\frac{N}{q}\right)S_{p,\alpha}^{-1}\|V\|_{\alpha} - S_{p,\alpha}^{-\frac{p-1}{p}}\|\tilde{W}\|_{\frac{\alpha p}{p-1}}\right)A_n\\
&-\left(\frac{N}{p}-\frac{N}{q}\right)B_n .
\end{aligned}
\end{equation}
In order to control the negative term involving $B_n$, for $1<s<p$, we use the interpolation inequality $$\|u_n\|^p_p\le\|u_n\|_s^{\beta p}\|u_n\|_q^{(1-\beta)p},\qquad\frac{1}{p}=\frac{\beta}{s}+\frac{1-\beta}{q},\,\beta\in(0,1),$$
and the Gagliardo-Nirenberg inequality we can see that,
\begin{equation}\notag
\begin{aligned}
&B_n\le \|u_n\|_s^{\beta p}\|u_n\|_q^{(1-\beta)p}
\le \kappa^{1-\beta} \|u_n\|_s^{\beta p}(\|u_n\|_s^{1-\theta}\|\nabla u_n\|_p^\theta)^{(1-\beta)p}
\\
&=\kappa_{\beta}\rho^{\beta p+(1-\theta)(1-\beta)p}A_n^{\theta (1-\beta)}
\le \kappa_{\beta}\rho^{\beta p+(1-\theta)(1-\beta)p}A_n,   
\end{aligned}
\end{equation}
for $\rho>0$ sufficiently small, since $(1-\beta)\theta\in(0,1)$ and $A_n\to\infty$ 
as $\rho\to 0$ uniformly in $n$. Moreover, due to \eqref{cond-norm-V}, we have 
$$\frac{N}{q}-\frac{N-p}{p}-\left(N-\frac{N}{s}+\frac{N}{q}\right)S_{p,\alpha}^{-1}\|V\|_{\alpha} - S_{p,\alpha}^{-\frac{p-1}{p}}\|\tilde{W}\|_{\frac{\alpha p}{p-1}}>0,$$
hence, using \eqref{A-n-bd-below}, due to \eqref{rel-lambda-n} we can see that there exists $\delta>0$ such that 
\begin{equation}
\rho^s\lambda_n\ge \delta g(\rho),\,\forall\,\rho\in(0,\bar{\rho}).
\end{equation}
As a consequence, recalling that $\lambda_n$ is bounded in $n$,
\begin{equation}\label{lim:lag:mult}
0<\rho^s\lambda_{\rho}:=\rho^s\liminf_n \lambda_n\le \limsup_n \rho^s\lambda_n<\infty,\qquad\forall \rho\in(0,\rho^*)    
\end{equation}
and $\rho^s\lambda_\rho\to\infty$ as $\rho\to 0^+$, which concludes the proof. 

\end{proof}

\subsection{Compactness}
Since the Palais-Smale sequence $u_n$constructed in Lemmata \ref{lemma-bdPS} and \ref{lemma-bd-lambda>0} is bounded in $W^{1,p}(\rn)$ and in $L^s(\rn)$ and the sequence of the Lagrange multipliers $\lambda_n$ fulfills \eqref{lambda-infty}, there exists $u=u_\rho\in W^{1,p}(\rn)$ such that up to a subsequence (relabelled by the same indices) $u_n\to u$ weakly in $W^{1,p}(\rn)$ and in $L^s(\rn)$ and $\lambda_n\to \lambda_\rho$. Due to the weak convergence, the pair $(u,\lambda)$ satisfies the equation
\begin{equation}
\label{eq-u}
-\Delta u+(\sgn(p-s)+V(x))|u|^{p-2}u+\lambda_\rho |u|^{s-2}u=|u|^{q-2}u\qquad\text{in}\,\rn.    
\end{equation}

If $\alpha\in(\frac{N}{p},\infty)$, the compactness of the bounded Palais-Smale sequence constructed in Lemmata \ref{lemma-bdPS} and \ref{lemma-bd-lambda>0} follows from the compactness of the embedding $W^{1,p}_r(\R^N)\subset L^t(\R^N)$ for $t\in(p,p^*)$, which holds thanks to the radial symmetry (see Lemma $1$ of \cite{dhara2025normalised}).\\

\begin{proof}[Proof of Theorem \ref{main-th-hom}, case $\alpha\in(\frac{N}{p},\infty)$]
Assume that $\{u_n\}\subset \mathcal{S}_{\rho,r}$ is the bounded Palais-Smale sequence of $J_V$ on $\mathcal{S}_{\rho,r}$ at level $c^r_{V,\rho}$ constructed in Lemmata \ref{lemma-bdPS} and \ref{lemma-bd-lambda>0}. In other words, $u_n$ satisfies
$$-\Delta_p u_n+(1+V(x))|u_n|^{p-2}u_n+\lambda_n |u_n|^{s-2}u_n=|u_n|^{q-2}u_n+o_n(1)\qquad\text{in}\,X^*,$$
where, up to a subsequence, $\lambda_n=\frac{\langle J_V'(u_n),u_n\rangle}{\rho^2}\to \lambda_\rho\in(0,\infty)$. Then there exists $u\in X$ such that, up to a subsequence, $u_k\rightharpoonup u$ weakly in $X$. As a consequence, $u_n\to u$ point-wise a.e. and hence $u\in X_r$. \\

Composing with the inverse $\Psi_{\lambda_\rho}$ constructed in Proposition \ref{prop-quasi-lin-pb}, we have $$u_n=\Psi_{\lambda_\rho}(|u_n|^{q-2}u_n-V|u_n|^{p-2}u_n+o_n(1))$$ 
Due to the compactness of the embedding $W^{1,p}_r(\R^N)\subset L^q(\R^N)$, for $q\in(p,p^*)$, we have $|u_n|^{q-2}u_n\to|u|^{q-2}u$ in $L^{q'}(\R^N)$ and, therefore, in $X^*$ too. Similarly, due to the compactness of the embedding $W^{1,p}_r(\R^N)\subset L^{\frac{p\alpha}{\alpha-1}}(\R^N)$, we have $V|u_n|^{p-2}u_n\to V|u|^{p-2}u$ in $L^{\frac{\alpha p}{\alpha -1}}(\R^N)$ and, therefore, in $X^*$ too. For this reason, using the continuity of $\Psi_{\lambda_\rho}$ we can see that $$u_n=\Psi_{\lambda_\rho}(|u_n|^{q-2}u_n-V|u_n|^{p-2}u_n+o_n(1))\to\Psi_{\lambda_\rho}(|u|^{q-2}u-V|u|^{p-2})\qquad\text{strongly in}\,X.$$
Since $u_n\rightharpoonup u$ weakly in $X$, we have $u=\Psi_{\lambda_\rho}(|u|^{q-2}u-V|u|^{p-2}u)$ and $u_k\to u$ strongly in $X$, hence we have concluded the proof of Theorem \ref{main-th-hom} in case $(V_1)$ holds with $\alpha\in (\frac{N}{p},\infty)$.
\end{proof}
We stress that the case $\alpha\in(\frac{N}{p},\infty)$ also includes the case $V\equiv 0$.\\
Now we consider the case $\alpha\in\{\frac{N}{p},\infty\}$. In this case, we need a version of the splitting Lemma for $s\in(1,p]$ (see Lemma \ref{splitting-lemma} below). This Lemma is a generalization of Lemma $2$ of \cite{dhara2025normalised}, which treats the case $s=2$.\\

Similarly to what happened in Lemma $2$ of \cite{dhara2025normalised}, in the proof of the splitting Lemma we need a lower bound for the $W^{1,p}(\R^N)$-norm for the solutions to the limit problem, that is problem \eqref{main-eq} with $V\equiv 0$. Such a lower bound is proved in \cite{dhara2025normalised} in the case $s=2$ and relies on the fact that the solutions to the limit problem for $s=2$ satisfy the Pohozaev identity \eqref{pohozaev-identity}. Here we extend this result to the case $s\in(1,p]$.\\

For $\rho>0$, let
\begin{equation}
\mathcal{P}_{V,\rho}:=\{u\in \mathcal{S}_\rho:\,P_V(u)=0\},  
\end{equation}
$P:=P_0,\,\mathcal{P}_{\rho}:=\mathcal{P}_{0,\rho}$ and $I:=J_0$.\\ 

We need the following characterization of $c_\rho=c^r_\rho$.
\begin{lemma}\label{lemma-characterization}
In the above notation,
\begin{enumerate}
\item $c_\rho=\inf_{\mathcal{P}_\rho}I(u)$.
\item There exists $\rho_0>0$ such that the function $\rho\in(0,\rho_0)\mapsto c_\rho$ is non-increasing.
\end{enumerate}
\end{lemma}
\begin{proof}
\begin{enumerate}
\item The proof follows the outlines of the proof of Lemma $2.3$ of \cite{WLZL}.
\item The proof follows the outlines of the proof of Lemma $2.8$ of \cite{WLZL}.
\end{enumerate}
\end{proof}
Using Lemma \ref{lemma-characterization}, we can prove that the solutions constructed in Theorem \ref{main-th-hom} for $V=0$ are actually the least-energy solution to \eqref{lim-eq}.
\begin{remark}\label{lemma-least-en}
Let $u\in\mathcal{S}_\rho$ be a solution to \eqref{lim-eq}, for some $\lambda>0$. Then, by Theorem \ref{cor-V>0}, $u\in \mathcal{P}_\rho$, so in particular $$I(u)\ge\inf_{v\in\mathcal{P}_\rho} I(v)=c_{\rho}.$$  
This yields that the solution $u_\rho\in \mathcal{S}_{\rho,r}$ constructed in Theorem \ref{main-th-hom} is the least energy solution to \eqref{lim-eq} in $\mathcal{S}_\rho$.
\end{remark}

Now we can prove the required lower bound for the $W^{1,p}(\R^N)$-norm of solutions to \eqref{lim-eq}.
\begin{lemma}\label{lemma-lower-bd}
Let $u\in X$ be a weak solution to the equation \eqref{lim-eq} with $0<\|u\|_s\le \rho$. Then $\|u\|\ge (pc_\rho)^{1/p}>0$.
\end{lemma}
\begin{proof}
Due to Lemma \ref{lemma-least-en}, we have $I(u)\ge c_{\bar{\rho}}$, where $\bar{\rho}:=\|u\|_s\in(0,\rho]$.
Hence we have
$$I(u)=\frac{1}{p}\|u\|^p-\frac{1}{q}\|u\|^q_q\ge \inf_{v\in\mathcal{P}_{\bar{\rho}}}I(v)=c_{\bar{\rho}}\ge c_\rho,$$
which concludes the proof.
\end{proof}
Now we are ready to state the splitting Lemma. In this Lemma, the potential is required to satisfy the same conditions as in Lemma $2$ of \cite{dhara2025normalised}, that is
\begin{itemize}
  \item[(i)] $V=V^+-V^-$, with $V^\pm\in L^{q^\pm}(\rn)$ for some $q^\pm\in\left[\frac{N}{p},\infty\right]$, $V^\pm\ge 0$, $V^+ V^-=0$.\label{V3}
 \color{black}\item[(ii)] $\lim_{|x|\to\infty}V^\pm(x)=0$.\label{V5}
\end{itemize}

\begin{lemma}[Splitting Lemma]
\label{splitting-lemma}
Let $N \geq 3$, $2 \color{red}\le \color{black}p < N$, $\lambda>0$, $q \in \bigl( \frac{p(N+s)}{N}, p^* \bigr)$ and $s\in(1,p]$.
Assume that $V$ satisfies $(i)-(ii)$. Let $\{u_n\} \subset W^{1,p}(\mathbb{R}^N) \cap L^s(\mathbb{R}^N)$ be a bounded Palais--Smale sequence for
\[
J_{V,\lambda}(u) = \frac{1}{p} \|u\|^p + \frac{1}{p} \int_{\mathbb{R}^N}V(x)|u|^p \, dx + \frac{\lambda}{s} \|u\|_s^s - \frac{1}{q} \|u\|_q^q
\]
such that $u_n \rightharpoonup u$ weakly in $W^{1,p}(\mathbb{R}^N) \cap L^s(\mathbb{R}^N)$. Then either $u_n \to u$ strongly in $W^{1,p}(\mathbb{R}^N) \cap L^s(\mathbb{R}^N)$, or there exists $k \geq 1$ non-trivial solutions $w^1, \dots, w^k \in W^{1,p}(\mathbb{R}^N) \cap L^s(\mathbb{R}^N)$ to the limit equation
\[
-\Delta_p w + \lambda |w|^{s-2}w + |w|^{p-2} w = |w|^{q-2} w \quad \text{in } \mathbb{R}^N,
\]
and $k$ sequences $\{y_n^j\} \subset \mathbb{R}^N$ satisfying $|y_n^j| \to \infty$ and $|y_n^j - y_n^i| \to \infty$ for $i \neq j$ such that, up to subsequence,
\[
u_n = u + \sum_{j=1}^k w^j(\cdot - y_n^j) + o_n(1) \quad \text{strongly in } W^{1,p}(\mathbb{R}^N) \cap L^s(\mathbb{R}^N).
\]
Moreover,
\begin{equation}\label{splitting-norm}
\|u_n\|_t^t = \|u\|_t^t + \sum_{j=1}^k \|w^j\|_t^t + o_n(1),\qquad\forall\, t\in\{s\}\cup[p,p^*)
\end{equation}
and
\begin{equation}\label{splitting-energy}
J_{V,\lambda}(u_n) = J_{V,\lambda}(u) + \sum_{j=1}^k I_\lambda(w^j) + o_n(1),
\end{equation}
where $I_\lambda(w) = \frac{1}{p} \|w\|^p + \frac{\lambda}{s} \|w\|_s^s - \frac{1}{q} \|w\|_q^q$.
\end{lemma}
Note that Lemma \ref{splitting-lemma} holds for any $s\in(1,p]$, not necessarily for $s=2$, which was the case for Lemma $2$ of \cite{dhara2025normalised}. Moreover, relation \eqref{splitting-norm} is more general than the corresponding relation in the statement Lemma $2$ of \cite{dhara2025normalised}, where only the case $t=s=2$ is treated. However, the technique in the proof is similar.
\begin{proof}
The proof follows step by step the proof of Lemma $2$ of \cite{dhara2025normalised}, with the only differnce that we apply the Brezis-Lieb Lemma to the $L^t(\R^N)$-norm for any $t\in\{s\}\cup[p,p^*)$. As in that proof, we need to use the fact that $\|w\|\ge(p c_\rho)^{1/p}$ for any weak solution $w\in X$ to \eqref{lim-eq} with $0<\|w\|_s\le \rho$, which is guaranteed, thanks to Lemma \ref{lemma-lower-bd}.
\newline
Let $\{u_n\} \subset W^{1,p}(\mathbb{R}^N) \cap L^s(\mathbb{R}^N)$ be a bounded Palais--Smale sequence for $J_{V,\lambda}$ such that $u_n \rightharpoonup u$ weakly in $W^{1,p}(\mathbb{R}^N) \cap L^s(\mathbb{R}^N)$.\\
{\bf Step 1. Weak Convergence:}
Define $\psi_n^1 := u_n - u$. Then, we have $\psi_n^1 \rightharpoonup 0$ in $W^{1,p}(\mathbb{R}^N) \cap L^s(\mathbb{R}^N)$. Therefore, up to a subsequence (still denoted by $\psi^1_n$), $\psi_n^1 \to 0$ in $L^t_{\mathrm{loc}}(\mathbb{R}^N)$ for $t \in [p, p^*)\cup\{s\}$ and point-wise a.e. in $\R^N$.
In fact, we can see by taking radial cutoff function $\varphi \in C^\infty_c(\mathbb{R}^N)$ such that $\varphi=1$ in $B_1$ and $\varphi\equiv 0$ in $\rn\setminus B_2$ (denoting $B_R(0)=B_R$), and $\varphi_R(x) = \varphi(x/R)$, we have
\[
\langle J'_{V,\lambda}(u_n),\varphi_R (u_n - u)\rangle = o_n(1),
\]
since $u_n$ is a Palais-Smale sequence of $J_{V,\lambda}$. 
Expanding the $\langle J'_{V,\lambda}(u_n),\varphi_R (u_n - u)\rangle$, we can see that
\begin{align*}
&\int_{B_{2R}} |\nabla u_n|^{p-2} \nabla u_n \cdot \nabla(u_n - u) \varphi_R \,dx+ \int_{B_{2R} \setminus B_R} |\nabla u_n|^{p-2} (u_n - u) \nabla u_n \cdot \nabla \varphi_R \,dx\\
&+ \int_{B_{2R}} |u_n|^{p-2} u_n (u_n - u) \Bigl({\rm sgn}(p-s) + V(x)\Bigr)\varphi_R\,dx+\lambda\int_{B_{2R}} |u_n|^{s-2} u_n (u_n - u) \varphi_R\,dx \\
&- \int_{B_{2R}} |u_n|^{q-2} u_n (u_n - u) \varphi_R\,dx = o_n(1).
\end{align*}
The gradient and nonlinearity terms are controlled as in the proof of Lemma $2$, p-7 of \cite{dhara2025normalised} using the facts from $(V_1)$ or $(V_2)$ and $\psi_n^1\to 0$ strongly in $L^{\bar{p}}_{\rm loc}(\rn),\ \bar{p}\in[p,p^*)$.\\
Thus, the same argument gives:
\[
\int_{B_{2R}(0)} |\nabla(u_n - u)|^p \varphi_R \, dx = o_n(1) \quad \Rightarrow \quad \nabla u_n \to \nabla u \text{ in } L^p_{\mathrm{loc}} \text{ and a.e.},
\]
so in particular, $\nabla u_n \to \nabla u$ pointwise a.e. in $\rn$.\\
{\bf Step 2. Brezis--Lieb and Vanishing of the potential term:}
As a consequence, by Brezis--Lieb lemma, we obtain the following splitting as $n \to \infty$:
\begin{align}
\|\nabla u_n\|_p^p &= \|\nabla u\|_p^p + \|\nabla \psi_n^1\|_p^p + o_n(1), \\
\|u_n\|_t^t &= \|u\|_t^t + \|\psi_n^1\|_t^t + o_n(1). \label{eq:s-split}
\end{align}
for any $t\in\{s\}\cup[p,p^*)$.\\
Since $\psi_n^1 \rightharpoonup 0$ weakly in $W^{1,p}(\rn)\cap L^s(\rn)$ and the embedding $W^{1,p}(\Omega)\subset L^t(\Omega)$ is compact, for any bounded domain $\Omega\subset\R^N$ and $t\in[p,p^*)$, $\psi_n^1 \to 0$ strongly in $L^t_{loc}(\rn)$ for such $t$ and pointwise a.e..
Therefore, in view of the assumptions $(V_1)$ 
(specifically that $V$ vanishes at infinity),
\begin{equation}\label{int-V-0}
\int_{\mathbb{R}^N} V(x)|\psi_n^1|^p \, dx \to 0 \quad \text{as } n \to \infty.    
\end{equation}
This fact can be proved as follows. First we note that, since $\lim_{|x|\to\infty} V(x)=0$, for any $\varepsilon>0$ there exists $R>0$ such that $|V(x)|\le \varepsilon$ if $|x|\ge R$. As a consequence, there exists $C>0$ independent of $\varepsilon$ and $n_0=n_0(\varepsilon)>0$ such that, for any $n\ge n_0$ and for any $M>1$, we have
\begin{equation}\notag
\begin{aligned}
\int_{\{V<M\}} V^+|\psi^1_n|^p dx&=\int_{\{V<M\}\cap B_R(0)} V^+|\psi^1_n|^p dx+\int_{\R^N\setminus B_R(0)} V^+|\psi^1_n|^p dx\\
&\le M\int_{B_R(0)}|\psi^1_n|^p dx+\varepsilon\int_{\R^N}|\psi^1_n|^p dx\le (M+C)\varepsilon.
\end{aligned}  
\end{equation}
Moreover, by the absolute continuity of the integral with respect to the measure of the integration set, there exists $\tilde{C}>0$ independent of $\varepsilon$ and $M_0=M_0(\varepsilon)>0$ such that, for any $M\ge M_0$ and for any $n$,
\begin{equation}\notag
\begin{aligned}
\int_{\{V>M\}} V^+|\psi^1_n|^p dx&\le \left(\int_{\{V>M\}} (V^+)^{q^+}dx\right)^{1/q^+}\left(\int_{\{V>M\}}|\psi^1_n|^{\frac{pq^+}{q^+-1}}\right)^{\frac{q^+-1}{pq^+}}\\
&\le \varepsilon\|\psi^1_n\|^{\frac{pq^+}{q^+-1}}_{\frac{pq^+}{q^+-1}}\le c\varepsilon \|\psi^1_n\|^{\frac{pq^+}{q^+-1}}\le \tilde{C}\varepsilon.   
\end{aligned}
\end{equation}
A similar argument applies to the term involving $V^-$. This concludes the proof of \eqref{int-V-0}.
Consequently, the mixed term involving the potential separates:$$\int_{\mathbb{R}^N} V(x)|u_n|^p \, dx = \int_{\mathbb{R}^N} V(x)|u|^p \, dx + \int_{\mathbb{R}^N} V(x)|\psi_n^1|^p \, dx + o_n(1) = \int_{\mathbb{R}^N} V(x)|u|^p \, dx + o_n(1).$$
{\bf Step 3. Decomposition of the Functional and Its Derivative:}
Using the norm splitting and the vanishing potential term, we can decompose the energy functional
$$J_{V,\lambda}(u_n) = J_{V,\lambda}(u) + I_\lambda(\psi_n^1) + o_n(1),$$
where $I_\lambda$ is the autonomous limit functional defined by
$$I_\lambda(w) = \frac{1}{p} \|w\|^p + \frac{\lambda}{s} \|w\|_s^s - \frac{1}{q} \|w\|_q^q.$$
Furthermore, since $u_n$ is a Palais--Smale sequence ($J'_{V,\lambda}(u_n) \to 0$) and $u$ is a critical point of $J_{V,\lambda}$ (by weak limit properties), standard arguments in critical point theory (see ~\cite[p:8-9]{dhara2025normalised}) imply that 
\begin{equation}\label{I'to0}
\langle I'_\lambda(\psi_n^1),\psi^1_n\rangle \to 0 \quad \text{as}\,n\to\infty.
\end{equation}
{\bf Step 4. Analysis of the Remainder:} Note that since $\psi_n^1$ is a Palais--Smale sequence for the autonomous functional $I_\lambda$.
Now if we assume that
$$\lim_{n\to\infty} \sup_{y\in\R^N}\int_{B_r(y)}|\psi^1_n|^pdx=0\qquad\forall r>0,$$
by Lemma I.$1$ of \cite{lions1984concentration2}, we have $\psi^1_n\to 0$ in $L^q(\R^N)$. 
Therefore, by \eqref{I'to0},
$$o_n(1)=\langle I'_\lambda(\psi^1_n),\psi^1_n\rangle=\|\psi^1_n\|^p+\lambda\|\psi^1_n\|^s_s+o_n(1),$$
which yields that $\psi^1_n\to 0$ in strongly $X$, hence the proof is over.\\

Otherwise there exists a subsequence, still denoted $\psi^1_n$, a sequence of points $\{y_n^1\} \subset \mathbb{R}^N$ and $r > 0$ such that
$$\liminf_{n \to \infty} \int_{B_r(y_n^1)} |\psi_n^1|^p \, dx > 0.$$
Thus, $|y_n^1| \to \infty$ because $\psi_n^1 \to 0$ strongly in $L^p_{loc}(\R^N)$. Define the shifted sequence $w_n^1(x) := \psi_n^1(x + y_n^1)$. Then $w_n^1 \rightharpoonup w^1$ weakly as $n\to\infty$, where $w^1 \not\equiv 0$ is a solution to the limit equation~\eqref{lim-eq}. Since $I_\lambda$ is translation invariant, $w^1$ is a critical point of $I_\lambda$, satisfying the limit equation
$$-\Delta_p w^1 + \lambda |w^1|^{s-2}w^1 +{\rm sgn}(p-s) |w^1|^{p-2} w^1 = |w^1|^{q-2} w^1.$$
{\bf Step 5. Iteration:} Define for $j\ge 2$, $$\psi_n^j(x) = \psi_n^j(x) - w^{j-1}(x - y_n^{j-1}) = u_n(x) - u(x) - w^{j-1}(x - y_n^{j-1}).$$
It is possible to find a sequence $y_n^j$ such that $|y_n^j|\to\infty$, $|y_n^j - y_n^i|\to\infty$ for $1\le i<j$ and $\psi_n^j(x)\rightharpoonup w^{j}(x-y_n^j(x))$ in $W^{1,p}(\rn)\cap L^s(\rn)$
where $w^j$ is a solution to~\eqref{lim-eq} and $y_n^j$ are translations. The iteration must terminate after a finite number of steps $k$. This is because each non-trivial solution $w^j$ carries a strictly positive amount of energy (bounded away from zero, see Lemma~\ref{lemma-lower-bd}), and the total energy $J_{V,\lambda}(u_n)$ is bounded.
Summing the decompositions yields
$$u_n = u + \sum_{j=1}^k w^j(\cdot - y_n^j) + o_n(1) \quad \text{strongly in } W^{1,p}(\rn) \cap L^s(\rn).$$
The norm and energy identities follow directly from the iterated Brezis-Lieb splittings
$$\|u_n\|_t^t = \|u\|_t^t + \sum_{j=1}^k \|w^j\|_t^t + o_n(1) \quad \forall t \in \{s\} \cup [p, p^*),$$
$$J_{V,\lambda}(u_n) = J_{V,\lambda}(u) + \sum_{j=1}^k I_\lambda(w^j) + o_n(1).$$
\end{proof}

Now we are ready to conclude the proof of Theorem \ref{main-th-hom} in the case $\alpha\in\{\frac{N}{p},\infty\}$.\\

\begin{proof}[Proof of Theorem \ref{main-th-hom}, case $\alpha\in\{\frac{N}{p},\infty\}$]
Let $\{u_n\}\subset X_r$ is the bounded Palais-Smale sequence constructed in Lemmata \ref{lemma-bdPS} and \ref{lemma-bd-lambda>0}. By the compactness of the embedding $W^{1,p}_{rad}(\R^N)\subset L^t(\R^N)$ for $t\in(p,p^*)$, proved in Lemma $1$ of \cite{dhara2025normalised}, we can see that, up to a subsequence, $u_n\to u$ strongly in $L^t(\R^N)$ for $t\in(p,p^*)$.\\

If we assume by contradiction that $u_n$ does not converge strongly to $u$ in $X$, the splitting Lemma \ref{splitting-lemma} yields that there exist $k\ge 1$ non-trivial weak solutions $w^1,\dots, w^k\in X\setminus\{0\}$ to
$$-\Delta_p w+{\rm sgn}(p-s)|w|^{p-2}w+\lambda |w|^{s-2}w=|w|^{q-2}w\qquad\text{in}\,\R^N$$
such that
$$u_n=u+\sum_{j=1}^k w^j(\cdotp-y^j_n)+o_n(1)$$
in $X$. Finally, by \eqref{splitting-norm}, we can see that 
$$\|u_n\|^t_t=\|u\|^t_t+\sum_{j=1}^k \|w^j\|^t_t+o_n(1),\qquad\forall\, t\in(p,p^*).$$
On the other hand, using that $u_n\to u$ in $L^t(\R^N)$, we have
$$\|u_n\|^t_t=\|u\|^t_t+o_n(1),$$
hence $w^j=0$ for any $j=1,\dots,k$, a contradiction.
\end{proof}

\bibliographystyle{plain}
\bibliography{ref}

\end{document}